\documentclass{amsart}
\usepackage{amscd}
\usepackage{amsmath,amssymb,amsfonts}
\usepackage{xypic}
\usepackage[mathcal]{euscript}
\usepackage[cmtip,all,2cell]{xy}
\usepackage{url}
\usepackage{latexsym}
\usepackage[all]{xy}
\usepackage{amsthm}
\usepackage[latin1]{inputenc}
\usepackage{multicol} 
\usepackage{geometry}
\usepackage{hyperref}

\UseAllTwocells

\newtheorem{theorem}{Theorem}[section]
\newtheorem{proposition}[theorem]{Proposition}
\newtheorem{lemma}[theorem]{Lemma}

\newtheorem{corollary}[theorem]{Corollary}

\theoremstyle{definition}
\newtheorem{definition}[theorem]{Definition}
\newtheorem{example}[theorem]{Example}

\newcommand{\nach}{\longrightarrow}
\newcommand{\auf}{\longmapsto}

\newcommand{\tfett}[1]{\textbf{#1}}

\newcommand{\ma}{\mathcal{A}}
\newcommand{\mc}{\mathcal{C}}
\newcommand{\md}{\mathcal{D}}
\newcommand{\me}{\mathcal{E}}
\newcommand{\mf}{\mathcal{F}}
\newcommand{\mm}{\mathcal{M}}
\newcommand{\nn}{\mathcal{N}}
\newcommand{\pp}{\mathcal{P}}

\newcommand{\Z}{\mathsf{Z}}
\newcommand{\SHC}{\mathsf{SHC}}

\DeclareMathOperator{\id}{id}

\DeclareMathOperator{\nat}{\mathsf{nat}}

\DeclareMathOperator{\Ho}{\mathsf{Ho}}

\DeclareMathOperator{\hhom}{\mathsf{hom}}
\DeclareMathOperator{\Hom}{\mathsf{Hom}}
\DeclareMathOperator{\Ext}{\mathsf{Ext}}
\DeclareMathOperator{\BiHom}{\mathsf{BiHom}}

\DeclareMathOperator{\END}{\mathsf{END}}

\DeclareMathOperator{\Sp}{\mathsf{Sp}}
\DeclareMathOperator{\Mod}{\mathsf{Mod}}

\DeclareMathOperator{\Ch}{\mathsf{Ch}}
\DeclareMathOperator{\Mon}{\mathsf{Mon}}

\DeclareMathOperator{\sSet}{\mathsf{sSet}}
\DeclareMathOperator{\Set}{\mathsf{Set}}

\DeclareMathOperator{\op}{op}

\DeclareMathOperator{\Cat}{\mathsf{Cat}}
\DeclareMathOperator{\CAT}{\mathsf{CAT}}

\DeclareMathOperator{\2Nat}{2-\mathsf{Nat}}
\DeclareMathOperator{\Dia}{\mathsf{Dia}}

\DeclareMathOperator{\pr}{pr}
\DeclareMathOperator{\DD}{\mathbb{D}}
\DeclareMathOperator{\EE}{\mathbb{E}}
\DeclareMathOperator{\FF}{\mathbb{F}}

\DeclareMathOperator{\LL}{\mathbb{L}}
\DeclareMathOperator{\RR}{\mathbb{R}}
\DeclareMathOperator{\SSS}{\mathbb{S}}

\DeclareMathOperator{\ipush}{i_{_\ulcorner}}

\DeclareMathOperator{\push}{\ulcorner}

\DeclareMathOperator{\inj}{\mathsf{inj}}
\DeclareMathOperator{\strict}{\mathsf{strict}}

\DeclareMathOperator{\PDer}{\mathsf{PDer}}
\DeclareMathOperator{\Der}{\mathsf{Der}}

\DeclareMathOperator{\PsNat}{\mathsf{PsNat}}

\DeclareMathOperator{\MonCAT}{\mathsf{MonCAT}}
\DeclareMathOperator{\MonPDer}{\mathsf{MonPDer}}

\DeclareMathOperator{\sMonCAT}{\mathsf{sMonCAT}}
\DeclareMathOperator{\sMonPDer}{\mathsf{sMonPDer}}
\DeclareMathOperator{\sMonDer}{\mathsf{sMonDer}}

\DeclareMathOperator{\Joyal}{\mathsf{Joyal}}
\DeclareMathOperator{\Kan}{\mathsf{Kan}}

\DeclareMathOperator{\sSetKan}{\sSet_{\Kan}}
\DeclareMathOperator{\sSetJoyal}{\sSet_{\Joyal}}

\DeclareMathOperator{\tw}{tw}
\DeclareMathOperator{\DistE}{\mathsf{Dist}_{\EE}}

\entrymodifiers={+!!<0pt,\fontdimen22\textfont2>}

\begin{document}

\title{Monoidal derivators and additive derivators}
\author{Moritz Groth}
\date{\today} 
\address{Moritz Groth, Radboud University, Nijmegen, Netherlands, email: m.groth@math.ru.nl}
\begin{abstract}
One aim of this paper is to develop some aspects of the theory of monoidal derivators. The passages from categories and model categories to derivators both respect monoidal objects and hence give rise to natural examples. We also introduce additive derivators and show that the values of strong, additive derivators are canonically pretriangulated categories. Moreover, the center of additive derivators allows for a convenient formalization of linear structures and graded variants thereof in the stable situation. As an illustration of these concepts, we discuss some derivators related to chain complexes and symmetric spectra.
\end{abstract}

\maketitle


\tableofcontents

\setcounter{section}{-1}

\section{Introduction}\label{p2section_intro}

The first main aim of this paper\footnote{This research was partially supported by the Deutsche Forschungsgemeinschaft within the graduate program `Homotopy and Cohomology' (GRK 1150)} is to develop some aspects of the theory of monoidal derivators (cf.\ also to \cite{cisinski_derivedkan}). As we saw in the companion paper \cite{groth_derivator}, two important classes of derivators are given by derivators associated to combinatorial model categories and derivators represented by bicomplete categories. Both classes of examples can be refined to give corresponding statements about situations where the `input is suitably monoidal'. We formalize the notion of a monoidal (pre)derivator and make these statements precise.

It is well-known that homotopy categories of (combinatorial) monoidal model categories (in the sense of Hovey~\cite{hovey}, as opposed to the slightly different notion of~\cite{SchwedeShipley_Alg}) can be canonically endowed with monoidal structures and similarly for suitably monoidal Quillen adjunctions. These statements are truncations of more structured results as we will see below. We will show that the derivator associated to a combinatorial monoidal model category can be canonically endowed with a monoidal structure. This easily generalizes to model categories which are suitable modules over a monoidal model category. However, to keep this paper at a reasonable length this will only be taken up in the sequel to this paper~(\cite{groth_enriched}) where we also talk about enriched derivators.

The second main aim of this paper is to discuss a few aspects of additive derivators. The author is not aware of a place in the literature where these derivators and linear structures on derivators are considered. It is for this purpose that we introduce additive derivators which by definition provide us with a class of derivators sitting between the pointed and the stable ones. Pushing further the corresponding result about stable derivators, we show that the values of a strong, additive derivator can be canonically turned into pretriangulated categories. Moreover, we also introduce the center of an additive derivator. This allows for a compact definition of a derivator which is linear over some commutative ring and also of graded variants thereof in the stable situation. We show that additive, monoidal derivators inherit canonically certain linear structures and there is a similar graded version for stable, monoidal derivators.

Since the passage from combinatorial model categories to derivators respects monoidal structures we obtain a conceptual explanation for the existence of linear structures on certain naturally occurring derivators.  As special cases we obtain, e.g., that the derivator of spectra is graded-linear over the stable homotopy groups of spheres and that the derivator of chain complexes is linear over the ground ring. It is easy to extend these two examples to modules over commutative monoids in which case we obtain graded-linear structures over the homotopy and homology ring of the commutative monoid respectively. 

We also have such a result for modules over non-commutative monoids. In these cases the derivators are graded-linear over the homotopy groups of the topological Hochschild cohomology of the ring spectrum and the Hochschild cohomology of the differential-graded algebra respectively. These last examples were our original motivation for studying monoidal and additive derivators but again will only be treated in the sequel \cite{groth_enriched}. In that paper we will also see that these (graded-) linear structures are only `shadows' of certain monoidal morphisms of monoidal derivators.

We now turn to a description of the content by sections.  In Section \ref{section_cartesian} we develop some general theory about the 2-category of (pre)derivators which is essential to the monoidal picture. Pre-derivators and derivators are respectively organized in a Cartesian monoidal 2-category. We then turn to the important related notion of a bimorphism and show that the bimorphism functor is corepresented by the Cartesian product. Once we have introduced bimorphisms which preserve homotopy colimits separately in each variable we can talk about adjunctions of two variables in the context of derivators. This latter notion is analyzed in some detail since it plays a key role both in later sections and in \cite{groth_enriched}. 

In Section \ref{section_monoidal} we consider the basic notions of monoidal (pre)derivators, monoidal morphisms, and monoidal transformations giving rise to corresponding 2-categories. We define a monoidal prederivator by making explicit the notion of a monoidal object in the Cartesian 2-category of prederivators. By definition a monoidal derivator is a derivator endowed with a monoidal structure such that the monoidal pairing preserves homotopy colimits separately in each variable. We then show that derivators associated to combinatorial monoidal model categories are canonically monoidal. More generally, we show that a Brown functor between model categories (cf.\ ~Definition \ref{definition_Brown}) induces a morphism of associated derivators and we mention some relevant examples. In the last subsection, given a monoidal derivator we associate a bicategory of distributors to it. This bicategory has some pleasant formal properties and encodes important structure. In particular, weighted homotopy (co)limits are subsumed by this structure as we discuss more carefully in \cite{groth_enriched}.

In Section \ref{section_additive} we define additive derivators by asking the underlying category to be additive. This implies that all values and all functors in sight are additive. Moreover, in the case of a strong, additive derivator we can recycle constructions from the stable context (\cite[Section 4]{groth_derivator}) in order to construct both a left and a right triangulation on its values. These fit together nicely in the sense of a pretriangulated structure. Moreover, this structure is shown to be compatible with suitably exact morphisms and hence, in particular, with the precomposition and homotopy Kan extension functors of the additive derivator. These results apply for example to the derivator of non-negative chain complexes over a ring (which is not stable!). We then introduce the center of an additive derivator and show how this leads to the notion of linear structures on additive derivators. These linear structures turn out to be levelwise linear structure compatible with all functors belonging to the derivator. Finally, we deduce that additive, monoidal are canonically linear over the ring of self-maps of the monoidal unit of the underlying monoidal category. There are also similar results for the graded variant of the center.

Before we begin with the proper content of this paper let us make two more comments. The first comment concerns set-theoretical issues. In what follows we will frequently consider the `category of categories' and similar gadgets. Strictly speaking there are some size issues which were to be considered here but these problems could be circumvented by a use of Grothendiecks language of universes. Since we do not wish to add an additional technical layer to the exposition by keeping track of the different universes we decided to ignore these issues. The second comment concerns duality. Many of the statements in this paper have dual statements which also hold true by the dual proof. The reason for this is that all classes of derivators under consideration are closed under the passage to the opposite derivators. In most cases, we will not make these statements explicit but nevertheless allow ourselves to refer to a statement also in cases where, strictly speaking, the dual statement is needed. 

\tfett{Acknowledgments.} It is a pleasure to thank Andr{\'e} Joyal, Stefan Schwede and Michael Shulman for fruitful discussions and their ongoing interest in the subject.

\section{The Cartesianness of the 2-category of derivators and related notions}\label{section_cartesian}

\subsection{The Cartesian monoidal 2-categories $\Der$ and $\PDer$}

In order to establish some notation we begin by quickly recalling the definitions of a prederivator and morphisms between prederivators (cf.\ \cite[Section 1 and 2]{groth_derivator} for more details). By contrast, we refer to \cite[Section 1]{groth_derivator} for the definition of a derivator, some motivation for the notion, and also important classes of examples. Other references for derivators include \cite{grothendieck,heller} and \cite{franke,keller_universal,maltsiniotis2,maltsiniotis1,cisinskineeman}. For the basic notions of the theory of 2-categories we refer to \cite{borceux1,maclane,kelly_2cat}. This language will be used slightly more systematically here than in the companion paper \cite{groth_derivator} but again nothing deep from that theory is used.

Let us recall that a \emph{prederivator} is a 2-functor $\DD\colon\Cat^{\op}\nach\CAT$ where $\Cat$ denotes the 2-category of small categories and $\CAT$ denotes the 2-category of (not necessarily small) categories. Spelling out this definition, we thus have for every small category $J$ an associated category $\DD(J),$ for a functor $u\colon J\nach K$ an induced functor $\DD(u)=u^\ast \colon\DD(K)\nach\DD(J)$ and for a natural transformation $\alpha\colon u\nach v$ of two such functors a natural transformation $\DD(\alpha)=\alpha^\ast\colon u^\ast\nach v^\ast$ as indicated in the following diagram:
$$
\xymatrix{
J \rtwocell^u_v {\alpha}& K,& & \DD(K) \rtwocell^{u^\ast}_{v^\ast}{\;\;\alpha^\ast}& \DD(J).
}
$$
These assignments are compatible with compositions and units in a strict sense, i.e., we have equalities of the respective expressions. One can of course also consider 2-functors which are only defined on certain 2-subcategories $\Dia\subseteq\Cat$ (for example finite categories, finite and finite-dimensional categories or posets) subject to certain closure properties (cf.\ Section 4 of \cite{groth_derivator}). This would then lead to the notion of a (pre)derivator \emph{of type} $\Dia.$ For simplicity, we will stick to the case of all small categories but everything that we do in this paper can also be done for prederivators of type $\Dia.$

A \emph{morphism} $F\colon\DD\nach\DD'$ \emph{of prederivators} is a pseudo-natural transformation of 2-functors (\cite[Definition 7.5.2]{borceux1}). Thus, such a morphism consists of a functor $F_J\colon\DD(J)\nach\DD'(J)$ for each small category~$J$ together with specified isomorphisms $\gamma^F_u\colon u^\ast\circ F_K\nach F_J\circ u^\ast$ for each functor $u\colon J\nach K.$ These isomorphisms have to be suitably compatible with compositions and identities. More precisely, given a pair of composable functors $J\stackrel{u}{\nach}K\stackrel{v}{\nach}L$ and a natural transformation $\alpha\colon u_1\nach u_2\colon J\nach K,$ we then have the following relation resp.\ commutative diagrams:
$$\xymatrix{
\gamma_{\id_J}=\id_{F_J}&&u^\ast v^\ast F\ar[r]^{\gamma _v}\ar@/_1.0pc/[dr]_{\gamma_{vu}}&u^\ast F v^\ast\ar[d]^{\gamma _u}&&u_1^\ast F\ar[r]^{\alpha^\ast}\ar[d]_{\gamma_{u_1}}&u_2^\ast F\ar[d]^{\gamma_{u_2}}\\
& &&  F u^\ast v^\ast && F u_1^\ast\ar[r]_{\alpha ^\ast}& Fu_2^\ast
}
$$
Here, we suppressed the indices of $F$ and the upper indices of the natural transformation $\gamma$ (as we will frequently do in the sequel) to avoid awkward notation. Moreover, we will not distinguish notationally between the natural transformations $\gamma$ and their inverses. If all the components $\gamma^F_u$ are identities then $F$ will be called a \emph{strict morphism.}

We will later introduce the notion of an adjunction of two variables between (pre)derivators and in that context it will be important that we also have a lax version of morphisms. So, let us call a lax natural transformation $F\colon\DD\nach\DD'$ a \emph{lax morphism of prederivators}. Thus, such a lax morphism consists of a similar datum as a morphism and satisfies the same coherence conditions with the difference that the natural transformations $\gamma^F_u\colon u^\ast \circ F\nach F\circ u^\ast$ are not necessarily invertible. For simplicity we will also apply the same terminology to `extranatural' variants thereof as in the context of adjunctions of two variables (cf.\ Lemma \ref{lemma_ad2var}).

Finally, let $F,\:G\colon\DD\nach\DD'$ be two morphisms of prederivators. A \emph{natural transformation} $\tau\colon F\nach G$ is a family of natural transformations $\tau_J\colon F_J\nach G_J$ which are compatible with the coherence isomorphisms belonging to the functors $F$ and $G.$ Thus, for every functor $u\colon J\nach K$ the following diagram commutes:
$$
\xymatrix{
u^\ast F\ar[r]^\tau\ar[d]_\gamma& u^\ast G\ar[d]^\gamma\\
Fu^\ast\ar[r]_\tau&Gu^\ast
}
$$ 
One checks that a natural transformation is precisely the same as a \emph{modification} of pseudo-natural transformations (see \cite[Definition 7.5.3]{borceux1}). Given two parallel morphisms $F$ and $G$ of pre-derivators let us denote by $\nat(F,G)$ the natural transformations from $F$ to $G.$ 

Thus, with prederivators as objects, morphisms as $1$-cells, and natural transformations as 2-cells we obtain the 2-category $\PDer$ of prederivators. In fact, this is just a special case of the 2-category of 2-functors, pseudo-natural transformations, and modifications. The full sub-2-category spanned by the derivators is denoted by $\Der.$ Given two (pre)derivators $\DD$ and $\DD'$ let us denote the category of morphisms by $\Hom(\DD,\DD')$ while we will write $\Hom^{\strict}(\DD,\DD')$ for the full subcategory spanned by the strict morphisms.

\begin{example}
The Yoneda embedding $y\colon\CAT\nach\PDer$ sends a category $\mc$ to the represented prederivator $y(\mc)\colon \Cat^{\op}\nach\CAT\colon J\auf \mc^J.$ Here, $\mc^J$ denotes the category of functors from~$J$ to~$\mc.$ The 2-categorical Yoneda lemma implies that for an arbitrary prederivator $\DD$ we have a natural isomorphism of categories
$$Y\colon\Hom_{\PDer}^{\strict}(y(J),\DD)\stackrel{\cong}{\nach}\DD(J).$$
For simplicity, we will sometimes drop the embedding $y$ from notation and again just write $\mc$ for the prederivator represented by a category $\mc$.
\end{example}

In every 2-category we have the notion of adjoint  $1$-morphisms, equivalences, and Kan extensions (see Sections 1 and 2 of \cite{street_formalmonads}). Since we will later introduce adjunctions of two variables let us recall the first notion in the 2-category $\Der$ of derivators (cf.\ \cite[Subsection 2.2]{groth_derivator}). A morphism $L\colon\DD\nach\DD'$ of derivators is a left adjoint if and only if it is levelwise a left adjoint functor $L_K\colon\DD(K)\nach\DD'(K)$ and it preserves homotopy left Kan extensions. 

For convenience, let us be more precise about the second condition. By definition of a derivator, the precomposition functor $u^\ast\colon\DD(K)\nach\DD(J)$ associated to a functor $u\colon J\nach K$ has adjoints on both sides which are called \emph{homotopy Kan extension functors} along~$u$. Thus, we have a \emph{homotopy left Kan extension functor} $u_!\colon\DD(J)\nach\DD(K)$ and a \emph{homotopy right Kan extension functor} $u_\ast\colon\DD(J)\nach\DD(K).$ Now, given a morphism $L\colon\DD\nach\DD'$ of derivators we can use these adjoints together with the adjunction morphisms in order to construct the Beck-Chevalley transformed 2-cells~${\gamma^L_u}_!$ (cf.\ \cite[Subsection 1.2]{groth_derivator}):
$$
\xymatrix{
\DD'(J)\xtwocell[1,1]{}\omit&\DD(J)\ar[l]_L&&
\DD'(K)\xtwocell[1,1]{}\omit&\DD'(J)\xtwocell[1,1]{}\omit\ar[l]_-{u_!} & \DD(J)\xtwocell[1,1]{}\omit \ar[l]_-L& &\\
\DD'(K)\ar[u]^{u^\ast}&\DD(K)\ar[l]^L\ar[u]_{u^\ast}&&
&\DD'(K) \ar@/^1.0pc/[lu]^-=\ar[u]^-{u^\ast} & \DD(K)\ar[l]^-L\ar[u]_-{u^\ast}& \DD(J)\ar@/_1.0pc/[lu]_-= \ar[l]^-{u_!}& 
}
$$
By definition, ${\gamma^L_u}_!\colon u_!\circ L\nach L\circ u_!$ is given by the pasting of the above right diagram in which the two additional natural transformations are adjunction morphisms. We say that~$L$ \emph{preserves homotopy left Kan extensions} if the natural transformation ${\gamma^L_u}_!\colon u_!\circ L\nach L\circ u_!$ is an isomorphism for all functors $u\colon J\nach K$.

Now, given such a left adjoint morphism of derivators $L\colon\DD\nach\DD'$ we choose levelwise right adjoint functors $R_K\colon\DD'(K)\nach\DD(K).$ These can be uniquely assembled into a morphism of derivators $R\colon\DD'\nach\DD$ such that the adjunctions at the different levels are compatible. By this we mean that for a functor $u\colon J\nach K$ we obtain the following commutative diagram
$$
\xymatrix{
\hhom_{\DD'(K)}(L_KX,Y)\ar[d]_{u^\ast}\ar[r]^\cong& \hhom_{\DD(K)}(X,R_KY)\ar[d]^{u^\ast}\\
\hhom_{\DD'(J)}(u^\ast L_KX,u^\ast Y)\ar[d]_{\gamma^L}&\hhom_{\DD(J)}(u^\ast X,u^\ast R_KY)\ar[d]^{\gamma^R}\\
\hhom_{\DD'(J)}(L_Ju^\ast X,u^\ast Y)\ar[r]_\cong&\hhom_{\DD(J)}(u^\ast X, R_Ju^\ast Y)
}
$$
where the morphisms $\gamma^L$ and $\gamma^R$ are the natural transformations which belong to the morphisms $L$ and $R$ respectively.

Let us define the (`internal') product of two prederivators. Thus, let $\DD$ and $\DD'$ be prederivators, then their product $\DD\times\DD'\in\PDer$ is defined to be the composition of 2-functors
$$
\xymatrix{
\Cat^{\op}\ar[r]^-\Delta&\Cat^{\op}\times\Cat^{\op}\ar[r]^-{\DD\times\DD'}&\CAT\times\CAT\ar[r]^-{\times}&\CAT
}
$$
where $\Delta$ denotes the diagonal. The product of morphisms of prederivators and natural transformations is defined similarly and this gives us the 2-product in the 2-category $\PDer$ of prederivators. Recall from \cite[Section 3 and 4]{groth_derivator} that we also have the notions of a pointed derivator and a stable derivator.

\begin{lemma}
Let $\DD$ and $\DD'$ be derivators. Then $\DD\times\DD'$ is again a derivator. Moreover, if $\DD$ and $\DD'$ are in addition pointed or stable then also $\DD\times\DD'$ is pointed or stable respectively.
\end{lemma}
\begin{proof}
Since isomorphisms in product categories are detected pointwise and since a product of two functors is an adjoint functor resp.\ an equivalence if and only if this is the case for the two factors the axioms (Der1)-(Der3) are immediate. Also the base change axiom holds since the base change morphism in $\DD\times\DD'$ can be taken to be the product of the base change morphisms in $\DD$ and $\DD'$ which are isomorphisms by assumption. Thus, with $\DD$ and $\DD'$ also the product $\DD\times\DD'$ is a derivator. Similarly, since the product of pointed categories is again pointed we obtain the result for pointed derivators. For stable derivators, note that $\DD\times\DD'$ is strong since the product of two full or essentially surjective functors is again full or essentially surjective respectively. Finally, an object $X=(Y,Y')\in\DD(\square)\times\DD'(\square)$ is (co)Cartesian if and only if the components $Y\in\DD(\square)$ and $Y'\in\DD'(\square)$ are (co)Cartesian. Hence, if $\DD$ and $\DD'$ are stable, the product $\DD\times\DD'$ is also stable.
\end{proof}

The product endows the 2-categories $\PDer$ and $\Der$ with the structure of a symmetric monoidal 2-category, called the Cartesian monoidal structure. The unit $e$ of the monoidal structure is the prederivator with constant value the terminal category $e$ (consisting of one object and its identity morphism only) and the symmetry constraint is given by the twist morphism $T\colon\DD\times\DD'\nach\DD'\times\DD$. To simplify notation we will suppress the canonical associativity isomorphisms and hence also brackets from notation. In the next section, we will introduce monoidal (pre)derivators as certain monoidal objects in the respective 2-categories. However, it turns out to be convenient to describe the monoidal structure using \emph{bimorphisms} which will be introduced in the next subsection.

\subsection{Bimorphisms and homotopy (co)limit preserving bimorphisms}

Since the product of two prederivators is the 2-categorical product we understand morphisms \emph{into} them. But also maps \emph{out of} a product of two prederivators are easy to describe: up to an equivalence of categories these are just the bimorphisms as we will define them now.

\begin{definition}
Let $\DD,\:\EE,$ and $\FF$ be prederivators. A \emph{bimorphism} $B$ \emph{from} $(\DD,\EE)$ \emph{to} $\FF$, denoted $B\colon(\DD,\EE)\nach\FF,$ consists of a family of functors
$$B_{J_1,J_2}\colon\DD(J_1)\times\EE(J_2)\nach\FF(J_1\times J_2),\qquad J_1,\:J_2\in\Cat,$$
and for each pair of functors $(u_1,u_2)\colon(J_1,J_2)\nach(K_1,K_2)$ a natural isomorphism $\gamma^B_{u_1,u_2}$ as indicated in:
$$
\xymatrix{
\DD(K_1)\times\EE(K_2)\xtwocell[1,1]{}\omit\ar[r]^-B\ar[d]_{u_1^\ast\times u_2^\ast} &\FF(K_1\times K_2)\ar[d]^{(u_1\times u_2)^\ast}\\
\DD(J_1)\times\EE(J_2)\ar[r]_-B&\FF(J_1\times J_2)
}
$$
These data have to satisfy the following coherence conditions. Given a pair of composable pairs $(u_1,u_2)\colon(J_1,J_2)\nach(K_1,K_2)$ and $(v_1,v_2)\colon(K_1,K_2)\nach(L_1,L_2)$ and a pair of natural transformations $(\alpha_1,\alpha_2)\colon(u_1,u_2)\nach(u_1',u_2')$ we have $\gamma_{\id_{J_1},\id_{J_2}}=\id_{B_{J_1,J_2}}$ and the commutativity of the following two diagrams:
$$
\xymatrix{
(u_1\times u_2)^\ast (v_1\times v_2)^\ast B\ar[r]^-\gamma \ar@/_1.0pc/[dr]_-\gamma&(u_1\times u_2)^\ast B (v_1^\ast\times v_2^\ast)\ar[d]^\gamma&
(u_1\times u_2)^\ast B\ar[r]\ar[d]_\gamma&(u_1'\times u_2')^\ast B\ar[d]^\gamma\\
&  B(u_1^\ast\times u_2^\ast)(v_1^\ast\times v_2^\ast) & 
B (u_1^\ast\times u_2^\ast)\ar[r]& B(u_1'^\ast \times u_2'^\ast)
}
$$
\end{definition}

Now, given two parallel bimorphisms $B,B'\colon(\DD,\EE)\nach\FF,$ a \emph{natural transformation} $\tau\colon B\nach B'$ \emph{of bimorphisms} consists of a family of natural transformations $\tau_{J_1,J_2}\colon B_{J_1,J_2}\nach B'_{J_1,J_2}.$ These have to be compatible in the sense that given a pair of functors $(u_1,u_2)\colon (J_1,J_2)\nach(K_1,K_2)$ the following diagram commutes:
$$
\xymatrix{
(u_1\times u_2)^\ast B\ar[r]^\tau\ar[d]_\gamma& (u_1\times u_2)^\ast B'\ar[d]^\gamma\\
B(u_1^\ast\times u_2^\ast)\ar[r]_\tau&B'(u_1^\ast\times u_2^\ast)
}
$$

Let us quickly mention that three of the above coherence properties can be expressed by equalities between certain pasting diagrams. This observation combined with the nice behavior of Beck-Chevalley transformation with respect to pasting (cf.\ \cite[Lemma 1.18]{groth_derivator}) proves to be useful in the discussion of adjunctions of two variables. 

Given three prederivators $\DD,\:\EE,$ and $\FF$ we obtain a category of bimorphisms from $(\DD,\:\EE)$ to $\FF$ which we denote by $\BiHom((\DD,\EE),\FF).$ In fact, given three such prederivators we can consider the \emph{exterior product} $\DD\boxtimes\EE$ \emph{of} $\DD$ and $\EE$ and the 2-functor $\FF\circ(-\times -)$ which are respectively defined by
$$(\DD\boxtimes\EE)(J_1,J_2)=\DD(J_1)\times\EE(J_2)\qquad\mbox{and}\qquad(\FF\circ(-\times-))(J_1,J_2)=\FF(J_1\times J_2).$$
Then, we have an equality of categories
$$\BiHom((\DD,\EE),\FF)=\PsNat(\DD\boxtimes\EE,\FF\circ(-\times -))$$
where $\PsNat(-,-)$ denotes the category of pseudo-natural transformations and modifications (cf.\ \cite[Definition 7.5.2, Definition 7.5.3]{borceux1}). This observation shows that $\BiHom((-,-),-)$ is functorial in all three arguments. 

Let us now show that $\BiHom((-,-),-)$ is corepresentable by the product. For prederivators~$\DD$ and~$\EE,$ the universal bimorphism $(\DD,\EE)\nach\DD\times\EE$ has components induced by the projections:
$$
\xymatrix{
\DD(J_1)\times\EE(J_2)\ar[rr]^-{\pr_1^\ast\times\pr_2^\ast}&&\DD(J_1\times J_2)\times\EE(J_1\times J_2)
}$$
This bimorphism gives the right adjoint in the following proposition.

\begin{proposition}\label{prop_bimorphism}
For prederivators $\DD,\:\EE,$ and $\FF$ we have natural isomorphisms between categories of strict (bi)morphisms and natural equivalences between categories of (bi)morphisms:
$$\BiHom^{\strict}((\DD,\EE),\FF)\stackrel{\cong}{\nach}\Hom^{\strict}(\DD\times\EE,\FF)\qquad\mbox{and}\qquad\BiHom((\DD,\EE),\FF)\stackrel{\simeq}{\nach}\Hom(\DD\times\EE,\FF)$$
\end{proposition}
\begin{proof}
We begin with the strict case and sketch a definition of functors in both directions which will be inverse to each other. Given a strict bimorphism~$B\colon(\DD,\EE)\nach\FF$ we obtain a strict morphism $l(B)\colon\DD\times\EE\nach\FF$ with components:
$$
\xymatrix{
l(B)_J\colon\DD(J)\times\EE(J)\ar[r]^-{B_{J,J}}&\FF(J\times J)\ar[r]^-{\Delta_J^\ast}&\FF(J)
}
$$
Conversely, given a strict morphism~$F\colon\DD\times\EE\nach\FF$ we can construct an associated bimorphism~$r(F)\colon(\DD,\EE)\nach\FF$ with components:
$$
\xymatrix{
r(F)_{J_1,J_2}\colon\DD(J_1)\times\EE(J_2)\ar[rr]^-{\pr_1^\ast\times\pr_2^\ast}&&
\DD(J_1\times J_2)\times\EE(J_1\times J_2)\ar[rr]^-{F_{J_1\times J_2}}&&\FF(J_1\times J_2)
}
$$

Thus, in the construction of the two functors we make use of the adjunction morphisms belonging to the adjunction $(\Delta,\times)\colon\Cat\rightharpoonup\Cat\times\Cat.$ In fact, the adjunction unit $\eta\colon\id\nach\times\circ\Delta$ is given by the diagonals while the adjunction counit $\epsilon\colon\Delta\circ\times\nach\id$ is given by the projections. Let us note that if we take this adjunction and pass to~$\Cat^{\op}$ and $\Cat^{\op}\times\Cat^{\op}$ respectively then the direction of~$\eta$ and~$\epsilon$ are inverted. 

Using $(\DD\boxtimes\EE)\circ\Delta=\DD\times\EE,$ we can thus rewrite the functor~$l$ as the following composition:
$$
\xymatrix{
\2Nat(\DD\boxtimes\EE,\FF\circ\times)\ar[r]^-{\circ\Delta}&\2Nat((\DD\boxtimes\EE)\circ\Delta,\FF\circ\times\circ\Delta)\ar[r]^-\eta& \2Nat(\DD\times\EE,\FF)
}
$$
\noindent
If we depict this construction graphically (and add artificially an identity 2-cell), it reminds us of the Beck-Chevalley transformation and looks like:
$$
\xymatrix{
\CAT&\CAT\ar[l]_-= & \Cat^{\op}\ar[l]_-\FF& &\\
&\CAT\ar[u]^-=\ar@/^1.0pc/[lu]^-= \xtwocell[-1,-1]{}\omit& \Cat^{\op}\times\Cat^{\op}\xtwocell[-1,-1]{}\omit\ar[l]^-{\DD\boxtimes\EE}\ar[u]_-\times& \Cat^{\op}\xtwocell[-1,-1]{}\omit\ar@/_1.0pc/[lu]_-= \ar[l]^-\Delta& 
}
$$
There is a similar reasoning for the functor~$r\colon\Hom(\DD\times\EE,\FF)\nach\BiHom((\DD,\EE),\FF).$ In fact,~$r$ is given by the composition:
$$
\2Nat((\DD\boxtimes\EE)\circ\Delta,\FF)\stackrel{\circ\times}{\nach}\2Nat((\DD\boxtimes\EE)\circ\Delta\circ\times,\FF\circ\times)\stackrel{\epsilon}{\nach}\2Nat(\DD\boxtimes\EE,\FF\circ\times)
$$
\noindent
Again, this can be depicted like a Beck-Chevalley transformation as we see in the next diagram:
$$
\xymatrix{
\CAT\xtwocell[1,1]{}\omit& \CAT\xtwocell[1,1]{}\omit\ar[l]_-= & \Cat^{\op}\times\Cat^{\op}\xtwocell[1,1]{}\omit \ar[l]_-{\DD\boxtimes\EE}& &\\
&\CAT \ar@/^1.0pc/[lu]^-=\ar[u]^-= & \Cat^{\op}\ar[l]^-\FF\ar[u]_-\Delta& \Cat^{\op}\times\Cat^{\op}\ar@/_1.0pc/[lu]_-= \ar[l]^-\times& 
}
$$

Let us next show that the composition $l\circ r$ is the identity. For this purpose let us consider the diagram on the left which depicts the value of $l\circ r$ applied to a morphism~$F\colon \DD\times\EE\nach\FF:$
$$
\xymatrix{
\CAT\xtwocell[1,1]{}\omit & \Cat^{\op}\times\Cat^{\op}\xtwocell[1,1]{}\omit \ar[l]_-{\DD\boxtimes\EE}&&  
\CAT& \Cat^{\op}\ar[l]_-\FF& \\
\CAT \ar[u]^-= & \Cat^{\op}\ar[l]^-\FF\ar[u]_-\Delta\xtwocell[1,1]{}\omit&\Cat^{\op}\times\Cat^{\op}\ar@/_1.0pc/[lu]_-= \ar[l]^-\times& 
\CAT\ar[u]^-=& \Cat^{\op}\times\Cat^{\op}\xtwocell[-1,-1]{}\omit\ar[l]^-{\DD\boxtimes\EE}\ar[u]_-\times& \Cat^{\op}\xtwocell[-1,-1]{}\omit\ar@/_1.0pc/[lu]_-= \ar[l]^-\Delta\\
&&\Cat^{\op}\ar[u]_\Delta\ar@/^1.0pc/[ul]^-=& 
&&\Cat^{\op}\times\Cat^{\op}\ar[u]_\times\ar@/^1.0pc/[ul]^-=\xtwocell[-1,-1]{}\omit 
}
$$
More precisely, this value is by definition the pasting of that diagram. The triangular identity for the adjunction~$(\Delta,\times)$ implies that this is just~$F$ as intended. A similar reasoning applies to the diagram on the right concluding the proof of the first statement.

Note that there is a subtlety in the above pasting diagrams. To make this precise let us consider the following commutative diagram describing that pasting and which again shows that~$l\circ r=\id.$ Using the labels of the arrows in that diagram, we have $(l\circ r)(F)=\eta\circ(F\circ\times\circ\Delta)\circ(\epsilon\circ\Delta):$
$$
\xymatrix{
(\DD\boxtimes\EE)\circ\Delta\ar[r]^-{\epsilon\circ\Delta}&(\DD\boxtimes\EE)\circ\Delta\circ\times\circ \Delta\ar[d]_{F\circ\times\circ\Delta}\ar[r]^-{\Delta\circ\eta}&
(\DD\boxtimes\EE)\circ\Delta\ar[d]^F\\
&\FF\circ\times\circ\Delta\ar[r]_-\eta&\FF
}
$$
The square is commutative in this case since~$F$ is a \emph{strict} morphism as opposed to a more general morphism of prederivators. In the case where~$F$ happens to be a morphism of prederivators this square would only commute up to an invertible 2-cell given by the structure maps $\gamma^F$ belonging to~$F.$ This is the reason why we only have an equivalence between the respective categories of (bi)morphisms and not an isomorphism in that case. We leave the details to the reader. 
\end{proof}

In the context of derivators, we want to introduce bimorphisms which preserve homotopy colimits separately in its arguments. For this purpose, let $B\colon(\DD,\EE)\nach\FF$ be a bimorphism of derivators and let us consider functors $u_1\colon J_1\nach K_1$ and $u_2\colon J_2\nach K_2.$ We can apply the formalism of Beck-Chevalley transformations to $\gamma^B_{u_1,\id}$ in order to obtain the natural transformation 
$${\gamma^B_{u_1,\id}}_!\colon (u_1\times \id)_!\circ B\nach B\circ ({u_1}_!\times \id)$$ 
as given by the following pasting:
$$
\xymatrix{
\FF(K_1\times J_2)\xtwocell[1,1]{}\omit&\FF(J_1\times J_2)\xtwocell[1,1]{}\omit\ar[l]_-{(u_1\times \id)_!} & \DD(J_1)\times\EE(J_2)\xtwocell[1,1]{}\omit \ar[l]_-B& &\\
&\FF(K_1\times J_2) \ar@/^1.0pc/[lu]^-=\ar[u]^-{(u_1\times\id)^\ast} & \DD(K_1)\times\EE(J_2)\ar[l]^-B\ar[u]_-{u_1^\ast\times\id}& \DD(J_1)\times\EE(J_2)\ar@/_1.0pc/[lu]_-= \ar[l]^-{{u_1}_!\times\id}& 
}
$$
A similar construction gives the natural transformation ${\gamma^B_{\id, u_2}}_!\colon (\id\times u_2)_!\circ B\nach B\circ(\id\times {u_2}_!).$

\begin{definition}\label{def_bicocont}
Let $\DD,\:\EE,$ and $\FF$ be derivators. A bimorphism $B\colon(\DD,\EE)\nach\FF$ \emph{preserves homotopy left Kan extensions in the first} or \emph{the second variable} if the natural transformations 
$${\gamma^B_{u_1,\id}}_!\colon (u_1\times \id)_!\circ B\nach B\circ ({u_1}_!\times \id)\qquad\mbox{or}\qquad {\gamma^B_{\id, u_2}}_!\colon (\id\times u_2)_!\circ B\nach B\circ(\id\times {u_2}_!)$$
are isomorphisms for all functors $u_1\colon J_1\nach K_1$ or all functors $u_2\colon J_2\nach K_2$ respectively.
\end{definition}

For simplicity we will also say that a morphism of derivators is \emph{cocontinuous in both variables} if it preserves homotopy left Kan extensions in the first and the second variable. This notion will be important in the context of adjunctions of two variables between derivators (cf.\ Subsection \ref{subsection_adjunctionsoftwo}).

Later, in the context of distributors, we will need homotopy coends (with parameters). So, let us give their construction and also establish the basic result which we refer to as `the Fubini theorem'. As an intermediate step, let us recall the notion of the \emph{twisted arrow category} associated to a category. Let $K\in\Cat$ and let us consider the associated functor $\hhom_K(-,-)\colon K^{\op}\times K\nach \Set.$ As a special case of a set-valued functor, $\hhom_K(-,-)$ has an associated category of elements which is a discrete Grothendieck opfibration over $K^{\op}\times K.$ The \emph{twisted arrow category of}~$K$ which will be denoted by $Ar(K)^{\tw}$ is the opposite of this category of elements. Thus, an object in this category is just a morphism $f\colon k_0\nach k_1$ while a morphism $f\nach f'$ is a commutative diagram as in:
$$
\xymatrix{
k_0\ar[r]^-{f}\ar[d]&k_1\\
k'_0\ar[r]_{f'}&k'_1\ar[u]
}
$$
This category comes equipped with a functor $(t,s)\colon Ar(K)^{\tw}\to K^{\op}\times K$ which sends an object $f\colon k_0\to k_1$ to~$(k_1,k_0)$. By construction, this functor is a discrete Grothendieck fibration. 

\begin{definition}
Let $\DD$ be a derivator and let $J,\:K,$ and $L$ be small categories. The \emph{homotopy coend functor} $\int^K\colon\DD(J\times K^{\op}\times K\times L)\nach \DD(J\times L)$ is defined by:
$$\int^K\colon\DD(J\times K^{\op}\times K\times L)\stackrel{(t,s)^\ast}{\nach}\DD(J\times Ar(K)^{\tw}\times L)\stackrel{{\pr}_!}{\nach}\DD(J\times L)$$
\end{definition}

Given a morphism of derivators $F\colon\DD\nach\DD'$ we can use the formalism of Beck-Chevalley transformations in order to obtain a canonical map
$$\int^KF(X)\nach F(\int^K X)$$
for $X\in\DD(J\times K^{\op}\times K\times L).$ Let us say that~$F$ \emph{preserves homotopy coends} if this map happens to be an isomorphism for all~$X$ and all~$K.$ We know from \cite[Proposition 2.4]{groth_derivator} that a morphism which preserves homotopy colimits already preserves homotopy left Kan extensions. This implies immediately the first statement of the next lemma.

\begin{lemma}\label{lemma_cocontcoend}
A homotopy colimit preserving morphism between derivators also preserves homotopy coends. In particular, homotopy coends are calculated pointwise. Similarly, a bimorphism which preserves homotopy colimits separately in each variable also preserves homotopy coends separately in each variable.
\end{lemma}
\noindent
The fact that homotopy coends are calculated pointwise can be suggestively written as follows. Given a derivator $\DD$, $X\in\DD(J\times K^{\op}\times K\times L),$ and objects $j\in J$ and $l\in L$ the canonical map
$$\int^KX(j,-,-,l)\nach (\int^K X)(j,l)$$
is an isomorphism.

In the context of derivators, the Fubini-type theorem about iterated homotopy coends takes the following form. 

\begin{lemma}\label{lemma_Fubini}
Let $\DD$ be a derivator, $t\colon J\times K^{\op}\times K\times L^{\op}\times L\times M\cong J\times (K\times L)^{\op}\times (K\times L)\times M$ the canonical isomorphism, and $X\in\DD(J\times (K\times L)^{\op}\times (K\times L)\times M).$ Then there are natural isomorphisms:
$$\int^K\int^L t^\ast X\cong \int^{K\times L}X\cong \int^L\int ^K t^\ast X$$
\end{lemma}
\begin{proof}
This follows immediately from the observation that there is a canonical isomorphism between $Ar(K\times L)^{\tw}$ and $Ar(K)^{\tw}\times Ar(L)^{\tw}$ which is compatible with the source and target maps in the sense that the following diagram commutes:
$$
\xymatrix{
Ar(K)^{\tw}\times Ar(L)^{\tw}\ar[d]\ar[r]& Ar(K\times L)^{\tw}\ar[d]\\
K^{\op}\times K\times L^{\op}\times L\ar[r]_-t&  K^{\op}\times L^{\op}\times K\times L
}
$$
\end{proof}

\subsection{Adjunctions of two variables}\label{subsection_adjunctionsoftwo}

Our next aim is to introduce the notion of an \emph{adjunction of two variables between (pre)derivators}. This will, in particular, allow us to talk about \emph{closed} monoidal derivators later.

We begin by recalling this notion from classical category theory. Let $\md,\:\me,$ and $\mf$ be categories and let us agree that we call a bifunctor $\otimes\colon\md\times\me\nach\mf$ a \emph{left adjoint of two variables} if there are functors $\Hom_l\colon\md^{\op}\times\mf\nach\me$ and $\Hom_r\colon\me^{\op}\times\mf\nach\md$ and natural isomorphisms as in:
$$
\hhom_{\mf}(X\otimes Y,Z)\quad\cong\quad\hhom_{\me}(Y,\Hom_l(X,Z))\quad\cong\quad \hhom_{\md}(X,\Hom_r(Y,Z))
$$
Interchangeably, we also say that the bifunctor~$\otimes$ is \emph{divisible on both sides} and we denote the `division functors' by $X\backslash Z=\Hom_l(X,Z)$ and $Z/Y=\Hom_r(Y,Z)$ respectively. 

Let us now go back to the context of (pre)derivators and let us consider three prederivators $\DD,\:\EE,$ and $\FF$ together with a bimorphism $\otimes\colon(\DD,\EE)\nach\FF.$ Moreover, let us assume that the functors
$\otimes\colon\DD(J_1)\times\EE(J_2)\nach\FF(J_1\times J_2)$
are divisible on both sides. Thus, for each pair of small categories $J_1,\:J_2\in\Cat$ we obtain functors: 
$$\Hom_l(-,-)\colon\DD(J_1)^{\op}\times\FF(J_1\times J_2)\nach \EE(J_2)\quad\mbox{and}\quad
\Hom_r(-,-)\colon\EE(J_2)^{\op}\times\FF(J_1\times J_2)\nach\DD(J_1)$$

\begin{lemma}\label{lemma_ad2var}
Let $\DD,\:\EE,$ and $\FF$ be prederivators and let $\otimes\colon(\DD,\EE)\nach\FF$ be a bimorphism which is levelwise divisible on both sides. Then there is a unique way to assemble any family of chosen adjoints $\Hom_l(-,-)\colon\DD(J_1)^{\op}\times\FF(J_1\times J_2)\nach \EE(J_2)$ into a \emph{lax} natural transformation $\Hom_l$ such that the following squares commute for all functors $u_1\colon J_1\nach K_1,$ $u_2\colon J_2\nach K_2,$ and objects $X\in\DD(K_1),\:Y\in\EE(K_2),\:Z\in\FF(K_1\times K_2)$:
$$
\xymatrix{
\hhom(X\otimes Y,Z)\ar[r]\ar[d]&\hhom(Y,\Hom_l(X,Z))\ar[dd]^{\gamma^{\Hom_l}_{u_1,\id}}\\
\hhom((u_1\times\id)^\ast(X\otimes Y),(u_1\times\id)^\ast Z)\ar[d]_{\gamma^\otimes_{u_1,\id}}&\\
\hhom(u_1^\ast X\otimes Y,(u_1\times\id)^\ast Z)\ar[r]&\hhom(Y,\Hom_l(u_1^\ast X,(u_1\times\id)^\ast Z))
}
$$
$$
\xymatrix{
\hhom(X\otimes Y,Z)\ar[r]\ar[d]&\hhom(Y,\Hom_l(X,Z))\ar[d]\\
\hhom((\id\times u_2)^\ast(X\otimes Y),(\id\times u_2)^\ast Z)\ar[d]_{\gamma^\otimes_{\id,u_2}}& \hhom(u_2^\ast Y,u_2^\ast\Hom_l(X,Z))\ar[d]^{\gamma^{\Hom_l}_{\id,u_2}}\\
\hhom(X\otimes u_2^\ast Y,(\id\times u_2)^\ast Z)\ar[r]& 
\hhom(u_2^\ast Y,\Hom_l(X,(\id\times u_2)^\ast Z))
}
$$
\end{lemma}
\begin{proof}
The uniqueness of the natural transformations $\gamma^{\Hom_l}_{u_1,\id}\colon\Hom_l(-,-)\nach\Hom_l(u_1^\ast,(u_1\times\id)^\ast)$ and $\gamma^{\Hom_l}_{\id,u_2}\colon u_2^\ast\Hom(-,-)\nach\Hom_l(-,(\id\times u_2)^\ast)$ follows by taking $Y=\Hom_l(X,Z)$ in either of the diagrams and tracing the adjunction counit through the diagrams. If we spell out what we obtain that way then we see that $\gamma^{\Hom_l}_{u_1,\id}$ is defined as the pasting
$$
\xymatrix{
\EE(K_2)&\FF(J_1\times K_2)\ar[l]_-{u_1^\ast X\backslash-} & \FF(K_1\times K_2) \ar[l]_-{(u_1\times\id)^\ast}&\\
&\EE(K_2) \ar@/^1.0pc/[lu]^-=\ar[u]_-{u_1^\ast X\otimes -}\xtwocell[-1,-1]{}\omit & \EE(K_2)\ar[l]^-=\ar[u]_-{X\otimes -}\xtwocell[-1,-1]{}\omit & \FF(K_1\times K_2)\ar@/_1.0pc/[lu]_-= \ar[l]^-{X\backslash-}\xtwocell[-1,-1]{}\omit 
}
$$
where $X$ lives in $\DD(K_1).$ Similarly, $\gamma^{\Hom_l}_{\id,u_2}$ is obtained by pasting the following diagram 
$$
\xymatrix{
\EE(J_2)&\FF(K_1\times J_2)\ar[l]_-{X\backslash-} & \FF(K_1\times K_2) \ar[l]_-{(\id\times u_2)^\ast}&\\
&\EE(J_2) \ar@/^1.0pc/[lu]^-=\ar[u]^-{X\otimes -}\xtwocell[-1,-1]{}\omit & \EE(K_2)\ar[l]^-{u_2^\ast}\ar[u]_-{X\otimes -}\xtwocell[-1,-1]{}\omit & \FF(K_1\times K_2)\ar@/_1.0pc/[lu]_-= \ar[l]^-{X\backslash-} 
}
$$
in which $X$ is an object of $\DD(K_1).$ Thus, the corresponding natural transformations are obtained by certain Beck-Chevalley transformations applied to the structure morphisms of $\otimes$. The good behavior of Beck-Chevalley transformations with respect to pasting (cf.\ \cite[Lemma 1.18]{groth_derivator}) and the triangular identities for adjunctions imply that the constructed natural transformations satisfy certain coherence conditions. But these coherence conditions are precisely the ones imposed on lax natural transformations as intended.
\end{proof}
\noindent
Again, the laxness in the statement refers to the fact that the natural transformations which were constructed are, in general, not invertible. We will come back to this in the context of derivators. And, of course, there is a similar result for arbitrary levelwise chosen adjoints $\Hom_r(-,-).$ 

In this lemma, we were very precise and split the construction of the structure morphisms in the two cases $\gamma^{\Hom_l}_{u_1,\id}$ and $\gamma^{\Hom_l}_{\id,u_2}.$ The point is that these two natural transformations play significantly different roles in the case of adjunctions of two variables between derivators as we want to discuss next. Motivated by the notion of a left adjoint morphism between derivators we give the following definition.

\begin{definition}
A bimorphism $\otimes\colon(\DD,\EE)\nach\FF$ between derivators is a \emph{left adjoint of two variables} (or \emph{divisible on both sides}) if $\otimes\colon\DD(J_1)\times\EE(J_2)\nach\FF(J_1\times J_2)$ is divisible on both sides for all $J_1,J_2\in\Cat$ and if $\otimes$ preserves homotopy left Kan extensions separately in each variable. 
\end{definition}

In the context of this definition, Lemma \ref{lemma_ad2var} guarantees that associated to the bimorphism $\otimes$ we can construct two lax transformations $\Hom_l(-,-)$ and $\Hom_r(-,-).$ Let us again focus on the case of $\Hom_l(-,-)$ but similar remarks apply to $\Hom_r(-,-).$ Using the explicit construction of $\gamma^{\Hom_l}_{\id,u_2}$ in the proof of that lemma and also \cite[Lemma 1.20]{groth_derivator} we see that the following natural transformations are conjugate: 
$$\gamma^{\Hom_l}_{\id,u_2}\colon u_2^\ast\Hom_l(-,-)\nach\Hom_l(-,(\id\times u_2)^\ast)\qquad \mbox{and}\qquad {\gamma^\otimes_{\id, u_2}}_!\colon (\id\times u_2)_!\circ \otimes\nach \otimes\circ(\id\times {u_2}_!)$$
In particular, our assumption that $\otimes$ preserves homotopy left Kan extensions in the second variable is equivalent to the fact that the structure morphisms $\gamma^{\Hom_l}_{\id,u_2}$ of $\Hom_l(-,-)$ are isomorphisms for all functors $u_2\colon J_2\nach K_2.$

Now, using a similar reasoning and again \cite[Lemma 1.20]{groth_derivator} we can deduce that the natural transformation $\gamma^{\Hom_l}_{u_1,\id}$ is conjugate to the following pasting
$$
\xymatrix{
\FF(K_1\times K_2)\xtwocell[1,1]{}\omit&\FF(J_1\times K_2)\xtwocell[1,1]{}\omit\ar[l]_-{(u_1\times\id)_!} & \EE(K_2)\xtwocell[1,1]{}\omit \ar[l]_-{u_1^\ast X\otimes -}&&\\
&\FF(K_1\times K_2) \ar@/^1.0pc/[lu]^-=\ar[u]^-{(u_1\times\id)^\ast} & \EE(K_2)\ar[l]^-{X\otimes -}\ar[u]_-= &\EE(K_2)\ar@/_1.0pc/[lu]_-= \ar[l]^-=& 
}
$$
In particular, $\gamma^{\Hom_l}_{u_1,\id}$ is an isomorphism if and only if the above map $(u_1\times\id)_!(u_1^\ast X\otimes Y)\nach X\otimes Y$ is an isomorphism. But, in general, there is no reason for this map being an isomorphism. Thus, the structure maps $\gamma^{\Hom_l}_{u_1,\id}$ are, in general, not invertible (we will see an example for this phenomenon at the end of this subsection). Let us collect these observations in the following lemma.

\begin{lemma}\label{lemma_ad2varDer}
Let $\otimes\colon(\DD,\EE)\nach\FF$ be a bimorphism between derivators which is levelwise divisible on both sides and let $\Hom_l$ and $\Hom_r$ be the lax transformation guaranteed by Lemma~\ref{lemma_ad2var}. Then~$\otimes$ preserves homotopy left Kan extensions in the second variable if and only if the structure maps $\gamma^{\Hom_l}_{\id,u_2}$ are invertible for all $u_2\colon J_2\nach K_2.$ Similarly, $\otimes$ preserves homotopy left Kan extensions in the first variable if and only if the structure morphisms $\gamma^{\Hom_r}_{u_1,\id}$ are invertible for all $u_1\colon J_1\nach K_1.$
\end{lemma}

We now turn to examples of adjunctions of two variables for derivators. Here, we will cover the examples of represented derivators. The case of Quillen adjunctions of two variables will be taken up in the next section. 

Let $\otimes\colon\mc\times\md\nach\me$ be a functor of two variables. We can extend $\otimes$ to a (strict) bimorphism $\otimes\colon(\mc,\md)\nach\me$ of the associated represented prederivators. In fact, for a pair of categories $(J_1,J_2)$ let us define $\otimes_{J_1,J_2}\colon\mc^{J_1}\times\md^{J_2}\nach\me^{J_1\times J_2}$ by sending a pair $(X,Y)$ to: 
$$\xymatrix{
X\otimes Y\colon J_1\times J_2\ar[r]^-{X\times Y}&\mc\times\md\ar[r]^-\otimes&\me
}
$$
Let us call this bimorphism $\otimes$ the \emph{bimorphism represented by} $\otimes.$

\begin{proposition}\label{prop_representedadjunction}
Let $\mc,\:\md$ be complete categories, $\me$ a category, and $\otimes\colon\mc\times\md\nach\me$ a left adjoint of two variables. The represented bimorphism $\otimes\colon(\mc,\md)\nach\me$ is then levelwise divisible on both sides. In particular, adjunctions of two variables between bicomplete categories induce adjunctions of two variables between represented derivators.
\end{proposition}
\begin{proof}
Let us content ourselves by showing that the represented bimorphism $\otimes\colon(\md,\me)\nach\mf$ is levelwise divisible on the left. Thus, we give the construction of $\Hom_l(-,-)$ and the natural isomorphism expressing one half of the fact that we have an adjunction of two variables. So, let us consider a pair of categories $(J_1,J_2)$ and let us construct a right adjoint 
$$\Hom_l(-,-)\colon{(\mc^{J_1})}^{\op}\times \me^{J_1\times J_2}\nach\md^{J_2}.$$
Using $(\mc^{J_1})^{\op}\cong(\mc^{\op})^{J_1^{\op}},$ as an intermediate step we can associate a pair $(X,Z)$ to the functor
$$\Hom_l(-,-)\circ(X\times Z)\colon J_1^{\op}\times J_1\times J_2\nach\mc^{\op}\times\me\nach\md.$$
Here, $\Hom_l\colon\mc^{\op}\times\me\nach\md$ is a functor expressing the fact that $\otimes$ is divisible on the left. Forming the end over the category $J_1$ we can define $\Hom_l(X,Z)\colon J_2\nach\md$ by:
$$\Hom_l(X,Z)(-)=\int_{j_1}\Hom_l(X(j_1),Z(j_1, -))$$
Let us check that this gives us the desired adjunction. For this purpose let us consider a functor $Y\in\md^{J_2}.$ Using the fact that natural transformations are obtained by a further end construction we can make the following calculation:
\begin{eqnarray*}
\hhom_{\me^{J_1\times J_2}}(X\otimes Y,Z)&\cong&\int_{(j_1,j_2)}\hhom_{\me}(X(j_1)\otimes Y(j_2),Z(j_1,j_2))\\
&\cong&\int_{(j_1,j_2)}\hhom_{\md}\big(Y(j_2),\Hom_l(X(j_1),Z(j_1,j_2))\big)\\
&\cong&\int_{j_2}\hhom_{\md}\big(Y(j_2),\int_{j_1}\Hom_l(X(j_1),Z(j_1,j_2))\big)\\
&=&\int_{j_2}\hhom_{\md}\big(Y(j_2),\Hom_l(X,Z)(j_2)\big)\\
&\cong&\hhom_{\md^{J_2}}\big(Y,\Hom_l(X,Z)\big)
\end{eqnarray*}
The third isomorphism follows from the Fubini-type theorem for ends and the fact that corepresented functors are end preserving, the second one is the adjunction isomorphism at the level of categories, while the first and the last one are given by the fact that natural transformations can be expressed as ends. This concludes the proof that the bimorphism $\otimes$ is levelwise divisible on both sides.  

The statement in the context of bicomplete categories can be proved in two ways. One way is to check that the bimorphism $\otimes\colon(\mc,\md)\nach\me$ is cocontinuous in both variables. By Lemma \ref{lemma_ad2varDer} we could equivalently show that the canonical map $u_2^\ast\Hom_l(X,Z)\nach\Hom_l(X,(\id\times u_2)^\ast Z)$ is an isomorphism. But this is true since ends with parameters are calculated pointwise.
\end{proof}

We use this example and the details of the proof to illustrate that the structure maps belonging to $\Hom_l(-,-)$ are not necessarily isomorphisms, i.e., that we only obtain \emph{lax} natural transformations as opposed to \emph{pseudo}-natural transformations. So, let us consider a functor $u_1\colon J_1\nach K_1,$ two diagrams $X\colon K_1\nach\mc$ and $Z\colon K_1\times J_2\nach\me$ and let us have a look at the diagram:
$$
\xymatrix{
\int_{k_1}\Hom_l\big(X(k_1),Z(k_1,-)\big)\ar@/^1.4pc/[rrd]^-{\pr_{u_1(j_1'),u_1(j_1'')}}\ar@{-->}[d]&&\\
\int_{j_1}\Hom_l\big(X(u_1(j_1)),Z(u_1(j_1),-)\big)\ar[rr]_-{\pr_{j_1',j_1''}}&&
\Hom_l\big(X(u_1(j_1')),Z(u_1(j_1''),-)\big)
}
$$
The upper left object is $\Hom_l(X,Z)$ and the lower left one is $\Hom_l(u_1^\ast X,(u_1\times \id)^\ast Z).$ The solid morphisms belong to the universal wedges of the respective end constructions. By the universal property of the lower wedge there is a unique dashed arrow as indicated which is compatible with all projection morphisms and this dashed arrow gives us
$$\gamma^{\Hom_l}_{u_1,\id}\colon \Hom_l(-,-)\nach\Hom_l(u_1^\ast(-),(u_1\times \id)^{\ast}(-)).$$  
To give a specific example such that $\gamma^{\Hom_l}_{u_1,\id}$ is not invertible let us consider the following situation. Let $\mc=\md=\me$ be $\Set,$ the category of sets, and let us take the adjunction of two variables given by the Cartesian closedness of $\Set.$ Moreover, let $J_1=J_2=e$ be the terminal category and let $u_1=k_1\colon e\nach K_1$ classify an object $k_1\in K_1.$ Then,~$X$ and~$Z$ would just be functors $K_1\nach\Set$ and the natural transformation $\gamma^{\Hom_l}_{u_1,\id}$ evaluated at~$X$ and~$Z$ is the evaluation map $k_1^\ast\colon\nat(X,Z)\nach \hhom_{\Set}(X(k_1),Z(k_1)$ which certainly is not an isomorphism in general.

\section{Monoidal derivators}\label{section_monoidal}

\subsection{Monoidal prederivators, monoidal morphisms, and monoidal transformations}\label{subsection_monoidalder}

Emphasizing similarity to the fact that a monoidal category (\cite{eilenberg_closed} or \cite{maclane}) is just a monoidal object (called a pseudo-monoid in \cite{day_hopfalgebroids}) in the Cartesian 2-category $\CAT,$ we could just say that a monoidal prederivator is a monoidal object in the Cartesian 2-category $\PDer$. We prefer to make this more explicit:

\begin{definition}
Let $\EE$ be a prederivator. A \emph{monoidal structure on} $\EE$ is a 5-tuple $(\otimes,\mathbb{S},a,l,r)$ consisting of two morphisms of prederivators
$$\otimes\colon\EE\times\EE\nach\EE\qquad\mbox{and}\qquad \mathbb{S}\colon e\nach\EE$$
and natural isomorphisms $l,\:a,$ and $r$ as indicated in the diagrams: 
$$
\xymatrix{
e\times\EE\ar[r]^{\SSS\times\id}\ar@/_0.8pc/[rd]_-\cong&\EE\times\EE\ar[d]^\otimes&&
\EE\times\EE\times\EE\ar[r]^-{\id\times\otimes}\ar[d]_{\otimes\times\id}\ar@{}[dr]|{\Longleftarrow}&\EE\times\EE\ar[d]^\otimes&&
\EE\times\EE\ar[d]_\otimes&\EE\times e\ar@/^0.8pc/[ld]^\cong\ar[l]_-{\id\times\SSS} \xtwocell[+1,-1]{}\omit\\
&\EE\xtwocell[-1,-1]{}\omit,&&
\EE\times\EE\ar[r]_\otimes&\EE,&&
\EE&
}
$$
This structure has to satisfy the usual coherence conditions. A \emph{symmetric monoidal structure on} $\EE$ is a $6$-tuple $(\otimes,\mathbb{S},a,l,r,t)$ where $(\otimes,\mathbb{S},a,l,r)$ is a monoidal structure and $t$ is a natural isomorphism as in 
$$
\xymatrix{
\EE\times\EE\ar[r]^T\ar@/_1.0pc/[dr]_\otimes\xtwocell[1,1]{}\omit&\EE\times\EE\ar[d]^\otimes\\
&\EE
}
$$
which satisfies additional coherence conditions as specified in \cite{maclane,borceux2}. A \emph{monoidal} resp.\ \emph{symmetric monoidal prederivator} is a prederivator endowed with a monoidal resp.\ symmetric monoidal structure.
\end{definition}

We will often denote a monoidal prederivator simply by $(\EE,\otimes,\SSS)$ or even by $\EE.$ This definition gives us the \emph{internal variant} of a monoidal prederivator. Using the concept of bimorphisms and the obvious generalizations to `more arguments' (which could be called \emph{trimorphisms} etc.) we can equivalently consider monoidal prederivators in their \emph{external variant}. Recall from Proposition~\ref{prop_bimorphism} that the product in~$\PDer$ corepresents the bimorphism functor. A similar result can be obtained for trimorphisms etc.\ in which case the adjunctions $(\Delta,\times)\colon\Cat\rightharpoonup\Cat^{\times n},\;n\geq 3,$ play a key role (see the proof of Proposition~\ref{prop_bimorphism}). As an upshot we deduce that a monoidal structure on a prederivator~$\EE$ can be equivalently given by a bimorphism $\otimes\colon(\EE,\EE)\to\EE$ and a unit morphism $\SSS\colon e\to\EE$ together with certain specified coherence isomorphisms. In the context of derivators this external variant is more convenient since we want to include certain compatibility conditions with respect to homotopy left Kan extensions (cf.\ Definition \ref{def_monder}).

The prederivator $e$ is the prederivator represented by the terminal category~$e$. So, the 2-categorical Yoneda lemma provides a natural isomorphism of categories
$$Y\colon\Hom_{\PDer}^{\strict}(e,\EE)\stackrel{\cong}{\nach}\EE(e).$$ 
Thus, in particular, a strict morphism $e\nach\EE$ amounts to the choice of an object in $\EE(e).$  A not necessarily strict morphism $e\nach\EE$ contains more information but see Lemma \ref{lem_strictunit} (this reflects the fact that we should work with the \emph{bicategorical} Yoneda lemma as opposed to the 2-\emph{categorical} one since we are working with pseudo-natural transformations instead of the more restrictive 2-natural transformations).

Let us briefly unravel the definition of a monoidal structure on a prederivator to convince ourselves that this actually is the notion we wanted to axiomatize. So, let $\EE$ be a (symmetric) monoidal prederivator and let $J$ be a category. Then, by definition, we have a functor $\otimes\colon\EE(J)\times\EE(J)\nach\EE(J),$ an object $\SSS(J)\in\EE(J),$ and also natural transformations which endow $\EE(J)$ with the structure of a (symmetric) monoidal category. Moreover, for a functor $u\colon J\nach K$ we have an induced natural isomorphism~$\gamma^\otimes_u$ as indicated in:
$$
\xymatrix{
\EE(K)\times\EE(K)\ar[r]^-\otimes\ar[d]_{u^\ast\times u^\ast}&\EE(K)\ar[d]^{u^\ast}\\
\EE(J)\times\EE(J)\ar[r]_-\otimes&\EE(J)\xtwocell[-1,-1]{}\omit
}
$$
Similarly, since $\mathbb{S}\colon e\nach\EE$ is a morphism of derivators we have a canonical natural isomorphism $\gamma^{\SSS}_u$ as in the following diagram:
$$
\xymatrix{
e\ar[r]^-{\SSS(K)}\ar@/_1.0pc/[dr]_{\SSS(J)}&\EE(K)\ar[d]^{u^\ast}\\
&\EE(J)\xtwocell[-1,-1]{}\omit
}
$$
It is easy to check that these two natural isomorphisms endow $u^\ast\colon\EE(K)\nach\EE(J)$ with the structure of a strong (symmetric) monoidal functor. For example, the definition of a natural transformation between morphisms of prederivators implies that the following diagram commutes:
$$
\xymatrix{
(\otimes\circ(\otimes\times\id))\circ u^\ast\ar[r]^a\ar[d]_\gamma&(\otimes\circ(\id\times\otimes))\circ u^\ast\ar[d]^\gamma\\
u^\ast\circ(\otimes\circ(\otimes\times\id))\ar[r]_a& u^\ast\circ (\otimes\circ(\id\times\otimes))
}
$$
Evaluating this at three objects $X,\:Y,$ and $Z\in\EE(K)$ gives us already one of the coherence conditions as imposed on a strong (symmetric) monoidal structure on a functor:
$$
\xymatrix{
(u^\ast X\otimes u^\ast Y)\otimes u^\ast Z\ar[r]^a\ar[d]_\gamma&u^\ast X\otimes (u^\ast Y\otimes u^\ast Z)\ar[d]^\gamma\\
u^\ast(X\otimes Y)\otimes u^\ast Z\ar[d]_\gamma&u^\ast X\otimes u^\ast(Y\otimes Z)\ar[d]^\gamma\\
u^\ast((X\otimes Y)\otimes Z)\ar[r]_a&u^\ast(X\otimes (Y\otimes Z))
}
$$
The other coherence axioms are checked similarly. Moreover, there is a corresponding result for natural transformations. Let $\alpha\colon u\nach v$ be a natural transformation of functors $J\nach K.$ Then it follows immediately that $\alpha^\ast\colon u^\ast\nach v^\ast$ is a monoidal transformation with respect to the canonical monoidal structures. For example the fact that $\mathbb{S}\colon e\nach \EE$ is a morphism of prederivators encodes that $\alpha^\ast$ is compatible with the unitality constraints of $u^\ast$ and $v^\ast.$ In fact, the commutative square on the left reduces to the triangle on the right:
$$
\xymatrix{
\SSS_Ju^\ast\ar[d]_{\alpha^\ast}\ar[r]^\gamma&u^\ast\SSS_K\ar[d]^{\alpha^\ast}&&
\SSS(J)\ar[r]\ar@/_0.8pc/[dr]&u^\ast(\SSS(K))\ar[d]^{\alpha^\ast}\\
\SSS_Jv^\ast\ar[r]_\gamma&v^\ast\SSS_K&&
&v^\ast(\SSS(K))
}
$$
Thus, a (symmetric) monoidal prederivator~$\EE$ factors canonically as
$$\EE\colon\Cat^{\op}\nach\MonCAT\nach\CAT\qquad\mbox{respectively}\qquad\EE\colon\Cat^{\op}\nach\sMonCAT\nach\CAT.$$
Here, $\MonCAT$ denotes the 2-category of monoidal monoidal categories with (strong) monoidal functors and monoidal transformations. Similarly, $\sMonCAT$ is the 2-category of symmetric monoidal categories. Note that the dual $\EE^{\op}$ of a monoidal prederivator $\EE$ is also canonically endowed with a monoidal structure. Before we turn to some interesting examples, let us quickly give the adapted classes of morphisms and natural transformations. Again, the same terminology will also apply for derivators.

\begin{definition}
Let $\EE$ and $\EE'$ be monoidal prederivators. A \emph{monoidal structure on a morphism} $F\colon\EE\nach\EE'$ \emph{of prederivators} is a pair of natural transformations
$$
\xymatrix{
\EE\times\EE\ar[r]^\otimes\ar[d]_{F\times F}&\EE\ar[d]^F&&
e\ar[r]^{\mathbb{S}}\ar@/_1.0pc/[rd]_-{\mathbb{S}}&\EE\ar[d]^F\\
\EE'\times\EE'\ar[r]_\otimes&\EE'\xtwocell[-1,-1]{}\omit&&
&\EE'\xtwocell[-1,-1]{}\omit
}
$$
such that the usual coherence conditions are satisfied. A monoidal structure is called \emph{strong} if these natural transformations are isomorphisms. A \emph{(strong) monoidal morphism} $F\colon\EE\nach\EE'$ between monoidal prederivators is a morphism endowed with a (strong) monoidal structure. 
\end{definition}

There is an obvious variant for the case of symmetric monoidal prederivators \cite{borceux2} which demands for an additional coherence property but which again will not be made precise. For completeness we include the definition of a monoidal natural transformation.

\begin{definition}
Let $\EE$ and $\EE'$ be monoidal prederivators and let $F,\;G\colon\EE\nach\EE'$ be monoidal morphisms. A natural transformation $\phi\colon F\nach G$ is called \emph{monoidal} if the following two diagrams commute:
$$
\xymatrix{
\otimes\circ (F\times F)\ar[r]\ar[d]_{\phi\times\phi}&F\circ\otimes\ar[d]^\phi&&\mathbb{S}\ar[r]\ar@/_1.0pc/[rd]&F\circ\mathbb{S}\ar[d]^\phi\\
\otimes \circ (G\times G)\ar[r]&G\circ \otimes&&&G\circ\mathbb{S}
}
$$
\end{definition}
\noindent
As in classical category theory, there is no additional assumption on a monoidal transformation of symmetric monoidal functors. Thus, with these notions we have the 2-categories of (symmetric) monoidal prederivators together with the strong monoidal morphisms and monoidal transformations, which are denoted by:
$$
\MonPDer\qquad\mbox{respectively}\qquad\sMonDer
$$

For a summary, let us use the following notation: Given a monoidal 2-category~$\mc,$ let us denote by $\Mon(\mc)$ the 2-category of monoidal objects in $\mc.$ For the case of the Cartesian monoidal 2-category $\CAT$ we have $\Mon(\CAT)=\MonCAT,$ the 2-category of monoidal categories. Thus, we may summarize our discussion by saying that there is the following isomorphism of 2-categories and there is an analogous variant for symmetric monoidal prederivators:
$$
\MonPDer=\Mon(\CAT^{\Cat^{\op}},\times,e)\cong\Mon(\CAT,\times,e)^{\Cat^{\op}}=(\MonCAT)^{\Cat^{\op}}
$$

Let us now turn to monoidal derivators. By definition we want to insist that the monoidal pairing preserves homotopy left Kan extensions separately in both variables (cf.\ Definition \ref{def_bicocont}). Theoretically speaking this rules out certain examples since not all monoidal pairings do have this property but most examples `showing up in nature' (in particular, the closed ones) are covered by this definition.

\begin{definition}\label{def_monder}
A \emph{monoidal derivator} $\EE$ is a derivator~$\EE$ endowed with a monoidal structure such that the monoidal pairing $\otimes\colon(\EE,\EE)\to\EE$ preserves homotopy left Kan extensions separately in each variable.
\end{definition}

In the companion paper \cite{groth_derivator} we saw that important classes of derivators are induced by bicomplete categories and model categories. These classes of example can be extended to the monoidal context. So, let us turn to prederivators represented by a category. Let us recall that given a small category~$J$ the value of the represented prederivator $\mc=y(\mc)$ is the functor category~$\mc^J$ of diagrams of shape~$J$ with values in~$\mc$.

\begin{example}\label{example_representedmonoidal}
The 2-functor $y\colon\CAT\nach\PDer$ sending a category $\mc$ to the represented prederivator $\mc$ preserves 2-products and hence monoidal objects. Thus, we obtain induced 2-functors
$$y\colon\MonCAT\nach\MonPDer\qquad\mbox{and}\qquad y\colon\sMonCAT\nach\sMonPDer.$$
In the internal variant, the monoidal structure on the prederivator represented by a monoidal category $\mc$ sends two objects $X,Y\in\mc^J$ to the composition
$$
\xymatrix{
J\ar[r]^-\Delta&J\times J\ar[r]^{X\times Y}&\mc\times\mc\ar[r]^-\otimes&\mc,
}
$$
where $\Delta$ is again the diagonal functor. The monoidal unit at level~$J$ is given by
$$
\xymatrix{
J\ar[r]^-p&e\ar[r]^-{\mathbb{S}}&\mm.
}
$$
where~$p$ is the unique functor to the terminal category. The monoidal prederivator $\mc=y(\mc)$ is a monoidal derivator if and only if~$\mc$ is bicomplete and the monoidal structure preserves colimits separately in each variable. There is a similar result for symmetric monoidal categories.
\end{example}

The second class of examples of monoidal derivators coming from combinatorial monoidal model categories will be treated in Subsection \ref{subsection_monoidalmodel}. We finish this subsection by giving a refinement of our last example for closed monoidal categories. Since in Subsection \ref{subsection_adjunctionsoftwo} we only discussed adjunctions of two variables for derivators we will again only consider that case.

\begin{definition}
A monoidal derivator~$\EE$ is \emph{biclosed} if the monoidal pairing $\otimes\colon(\EE,\EE)\to\EE$ is a left adjoint of two variables. A symmetric monoidal derivator having this additional property is called a \emph{closed symmetric monoidal derivator.}
\end{definition}

\begin{corollary}
Let $\mc$ be a (bi)closed monoidal, bicomplete category then the associated represented derivator is (bi)closed monoidal.
\end{corollary}

\subsection{Monoidal model categories induce monoidal derivators}\label{subsection_monoidalmodel}

Before we turn to monoidal model categories let us recall in some detail how the the derivator $\DD_{\mm}$ associated to a combinatorial model category $\mm$ (cf.\ \cite[Proposition 1.36]{groth_derivator}) is constructed. The point of this slightly lengthy discussion is that it prepares the proof that Brown functors between model categories induce morphisms at the level of derivators (Proposition \ref{prop_brown}).

Recall that combinatorial model categories as introduced by Smith are cofibrantly generated model categories which have an underlying presentable category (for the theory of presentable categories cf.\ the original source \cite{gabrielulmer} but also \cite{adamekrosicky,makkai}). In the construction of the derivator~$\DD_{\mm}$ we use the fact that the diagram categories $\mm^J$ associated to such a model category~$\mm$ can be endowed both with the injective and the projective model structure. The existence of the projective model structure follows from a general lifting result of cofibrantly generated model structures along a left adjoint functor (\cite{hirschhorn}) while the existence of the injective model structure is, for example, shown in \cite[Proposition A.2.8.2]{HTT}. Since both model structures have the same class of weak equivalences, it is not important which one we use in the definition of the value $\DD_{\mm}(J)$ as they have canonically isomorphic homotopy categories:
$$\Ho(\mm_{proj}^J)\cong \Ho(\mm_{inj}^J)$$
Now, for a functor $u\colon J\nach K,$ the induced precomposition functor $u^\ast\colon \mm^K\nach\mm^J$ preserves weak equivalences with respect to both structures. Hence, by the universal property of the localization functor $\gamma\colon\mm^K\nach\Ho(\mm^K)$ we obtain a unique induced functor $u^\ast$ at the level of homotopy categories such that the following diagram commutes on the nose:
$$
\xymatrix{
\mm^K\ar[r]^{u^\ast}\ar[d]_{\gamma}&\mm^J\ar[d]^\gamma\\
\Ho(\mm^K)\ar[r]_{u^\ast}&\Ho(\mm^J)
}
$$
By definition, this induced functor is taken as the value $\DD_{\mm}(u).$ The fact that the localization $\gamma\colon\mm^K\nach\Ho(\mm^K)$ is a 2-localization allows us to complete the definition of~$\DD_{\mm}$ and shows that we have an actual 2-functor. Moreover, this universality in the 2-categorical sense implies that the left or right derived functors of a Quillen functor that preserves weak equivalences is isomorphic to its unique extension. Thus, in our context we obtain the following natural isomorphisms:
$$\LL\!u^\ast\quad\cong\quad u^\ast\quad\cong\quad\RR\!u^\ast.$$

This observation will be useful in the construction of the monoidal derivator underlying a combinatorial monoidal model category. More generally, it allows for the construction of morphisms of derivators induced by Brown functors and hence, in particular, by Quillen functors or Quillen bifunctors. One motivation for the notion of Brown functors is the following. In order to form the derived functor of a --say-- left Quillen functor not all of the defining properties of a left Quillen functor are needed as already emphasized in \cite{hovey,hirschhorn,maltsiniotis_quillen}. Thus, sometimes the following definition is useful (cf.\ also to \cite{dhks_homotopy} and \cite{shulman_composites} where these are called deformable functors and derivable functors, respectively).

\begin{definition}\label{definition_Brown}
Let $\mm$ and $\nn$ be model categories and let $F\colon\mm\nach\nn$ be a functor. $F$ is a \emph{left Brown functor} if $F$ preserves weak equivalences between cofibrant objects. Dually, $F$ is a \emph{right Brown functor} if $F$ preserves weak equivalences between fibrant objects.
\end{definition}
\noindent
As one sees from the constructions in \cite{hovey, hirschhorn}, this suffices to obtain the respective derived functors which again will have the universal property of the respective Kan extensions. In what follows, we will only state and prove the results for left Brown functors (and left Quillen (bi)functors), but also the dual statements hold true.

\begin{proposition}\label{prop_brown}
Let $\mm$ and $\nn$ be combinatorial model categories and let $F\colon\mm\nach\nn$ be a left Brown functor. Then by forming left derived functors we obtain a morphism of derivators $\LL\! F\colon\DD_{\mm}\nach\DD_{\nn}.$ In particular, this is the case for left Quillen functors.
\end{proposition}
\begin{proof}
Let $J$ be a category and let us consider the induced functor $F\colon\mm^J\nach\nn^J.$ With respect to the injective model structures, this is again a left Brown functor. Hence, given a functor $u\colon J\nach K$ we have the following commutative diagram on the left consisting of left Brown functors only:
$$
\xymatrix{
\mm_{\inj}^K\ar[r]^-F\ar[d]_{u^\ast}&\nn_{\inj}^K\ar[d]^{u^\ast}&&
\DD_{\mm}(K)\ar[d]_{u^\ast}\ar[r]^-{\LL\!F}\xtwocell[1,1]{}\omit&\DD_{\nn}(K)\ar[d]^{u^\ast}\\
\mm_{\inj}^J\ar[r]_-F&\nn_{\inj}^J&&
\DD_{\mm}(J)\ar[r]_-{\LL\!F}&\DD_{\nn}(J)
}
$$
Passing to left derived functors for the horizontal arrows and to the induced functors on the localizations for the vertical arrows gives us the diagram on the right which by our above discussion commutes up to a canonical natural isomorphism $\gamma_u$. It is easy to check that these natural isomorphisms~$\gamma_u,$ $u\colon J\nach K,$ endow the functors $\LL\!F$ with the structure of a morphism of derivators.
\end{proof}

\begin{corollary}
Let $(F,U)\colon\mm\rightharpoonup\nn$ be a Quillen adjunction of combinatorial model categories. Then we obtain a derived adjunction $(\LL\! F,\RR\! U)\colon\DD_{\mm}\rightharpoonup\DD_{\nn}.$ If $(F,U)$ happens to be a Quillen equivalence then $(\LL\! F, \RR\! U)$ is an equivalence of derivators.
\end{corollary}
\begin{proof}
For the case of adjunctions it suffices to observe that we obtain morphisms of derivators~$\LL\! F$ and~$\RR\! U$ which are levelwise adjoint in a compatible way. Thus, \cite[Proposition 2.11]{groth_derivator} implies that we have an adjunction. Alternatively, one could check that~$\LL\! F$ preserves homotopy left Kan extensions.
The case of equivalences is even easier since these are detected pointwise (see again \cite[Proposition 2.11]{groth_derivator}).
\end{proof}

There is a further important class of Brown functors, namely the Quillen bifunctors. These are central to many notions of homotopical algebra.

\begin{definition}
Let $\mm,\;\nn,$ and $\pp$ be model categories. A functor $\otimes\colon\mm\times\nn\nach\pp$ is a \emph{left Quillen bifunctor} if it preserves colimits separately in each variable and has the following property: For every cofibration $f\colon X_1\nach X_2$ in $\mm$ and every cofibration $g\colon Y_1\nach Y_2$ in $\nn$ the pushout-product map 
$$f\Box g=(X_2\otimes g)\amalg (f\otimes Y_2)\colon\quad X_2\otimes Y_1\amalg_{X_1\otimes Y_1}X_1\otimes Y_2\nach X_2\otimes Y_2$$ 
is a cofibration which is acyclic if in addition $f$ or $g$ is acyclic.
\end{definition}

There is the dual notion of a right Quillen bifunctor $\Hom\colon\mm^{op}\times\nn\nach\pp$. In that case one considers the induced maps 
$$\Hom_\Box(f,g)\colon\Hom(X_2,Y_1)\nach\Hom(X_1,Y_1)\times_{\Hom(X_1,Y_2)}\Hom(X_2,Y_2).$$
The following is immediate.

\begin{lemma}
Let $\otimes\colon\mm\times\nn\nach\pp$ be a left Quillen bifunctor and let $X\in\mm$ resp.\ $Y\in\nn$ be cofibrant objects. The functors $X\otimes -\colon\nn\nach\pp$ and $-\otimes Y\colon\mm\nach\pp$ are then left Quillen functors. In particular, $\otimes\colon\mm\times\nn\nach\pp$ is a left Brown functor when we endow $\mm\times\nn$ with the product model structure.
\end{lemma}

Thus Proposition \ref{prop_brown} can be applied to Quillen bifunctors. Under the canonical isomorphism $\DD_{\mm\times\nn}\cong\DD_{\mm}\times\DD_{\nn}$ we obtain that a Quillen bifunctor $\otimes\colon\mm\times\nn\nach\pp$ induces a morphism of derivators $\DD_{\mm}\times\DD_{\nn}\nach\DD_{\pp}.$ Let us not distinguish notationally between this morphism and the associated bimorphism (cf.\ Proposition \ref{prop_bimorphism}) and let us denote both by
$$
\stackrel{\LL}{\otimes}\colon\DD_{\mm}\times\DD_{\nn}\nach\DD_{\pp}\qquad\mbox{and}\qquad \stackrel{\LL}{\otimes}\colon(\DD_{\mm},\DD_{\nn})\nach\DD_{\pp}.
$$
The bimorphism can also be obtained without invoking Proposition \ref {prop_bimorphism}. The bifunctor $\otimes$ induces a strict bimorphism of represented derivators $\otimes\colon(\mm,\nn)\nach\pp.$ For each morphism of pairs $(u_1,u_2)\colon(J_1,J_2)\nach(K_1,K_2)$ we have a commutative diagram of left Brown functors as follows if all model categories are endowed with the injective model structures:
$$
\xymatrix{
\mm^{K_1}\times\nn^{K_2}\ar[r]^-\otimes\ar[d]_{u_1^\ast\times u_2^\ast}&\pp^{K_1\times K_2}\ar[d]^{(u_1\times u_2)^\ast}\\
\mm^{J_1}\times\nn^{J_2}\ar[r]_-\otimes&\pp^{J_1\times J_2}
}
$$
Forming derived functors at the different levels and taking the natural isomorphisms induced by these diagrams we obtain again the bimorphism $(\DD_{\mm},\DD_{\nn})\nach\DD_{\pp}.$

In the context of combinatorial model categories, we get a stronger statement. Recall that the adjoint functor theorem of Freyd takes the following form in the context of presentable categories: a functor between presentable categories is a left adjoint if and only if it preserves colimits. For example, in the context of combinatorial model categories a monoidal structure which preserves colimits in each variable is always a biclosed monoidal structure, i.e., we have an adjunction of two variables $(\otimes,\Hom_l,\Hom_r).$

Now, let $\mm,\;\nn,$ and $\pp$ be combinatorial model categories. Then given a left Quillen bifunctor $\otimes\colon\mm\times\nn\nach\pp$ we obtain an adjunction of two variables $(\otimes, \Hom_l,\Hom_r).$ This adjunction is expressed by natural isomorphisms 
$$\hhom_{\pp}(X\otimes Y,Z)\quad\cong \quad\hhom_{\mm}(X,\Hom_r(Y,Z)) \quad\cong\quad \hhom_{\nn}(Y,\Hom_l(X,Z))$$
for certain functors
$$\Hom_l(-,-)\colon\mm^{\op}\times\pp\nach\nn\qquad\mbox{and}
\qquad\Hom_r(-,-)\colon\nn^{\op}\times\pp\nach\mm.$$

\begin{lemma}
Let $\mm,\;\nn,$ and $\pp$ be model categories and let $(\otimes,\Hom_l,\Hom_r) \colon\mm\times\nn\rightharpoonup\pp$ be an adjunction of two variables. If we endow $\mm^{\op}$ resp.\ $\nn^{\op}$ with the dual model structures we have the following equivalent statements: $\otimes$ is a left Quillen bifunctor if and only if $\Hom_l$ is a right Quillen bifunctor if and only if $\Hom_r$ is a right Quillen bifunctor.
\end{lemma}
\noindent

By the above discussion, we know that a left Quillen bifunctor $\otimes\colon\mm\times\nn\nach\pp$ between combinatorial model categories extends to an adjunction of two variables. By Proposition \ref{prop_representedadjunction} or again by the special adjoint functor theorem, we deduce that this adjunction induces an adjunction of two variables between represented derivators $\otimes\colon(\mm,\nn)\nach\pp.$ By the last lemma, we have thus adjunctions of two variables consisting of Quillen bifunctors which induce derived adjunctions of two variables $\DD_{\mm}(J_1)\times\DD_{\nn}(J_2)\nach\DD_{\pp}(J_1\times J_2).$ Now, using Lemma \ref{lemma_ad2varDer}, we could proceed in two possible ways. Either we check that the bimorphism $\stackrel{\LL}{\otimes}\colon(\DD_{\mm},\DD_{\nn})\nach\DD_{\pp}$ preserves homotopy colimits separately in each variable or we show that certain structure maps belonging to $\RR\!\Hom_l$ and $\RR\!\Hom_r$ are isomorphisms. In both cases, our conclusion is that the bimorphism 
$$\stackrel{\LL}{\otimes}\colon(\DD_{\mm},\DD_{\nn})\nach\DD_{\pp}$$ 
is a left adjoint of two variables. We have thus established the following result.

\begin{corollary}\label{cor_derivedadjunction}
Let $\mm,\;\nn,$ and $\pp$ be combinatorial model categories and let $\otimes\colon\mm\times\nn\nach\pp$ be a left Quillen bifunctor. Then, by forming derived functors, we obtain an adjunction of two variables at the level of associated derivators:
$$(\stackrel{\LL}{\otimes},\RR\!\Hom_l,\RR\!\Hom_r)\colon(\DD_{\mm},\DD_{\nn})\rightharpoonup\DD_{\pp}$$
\end{corollary}

For later reference let us quickly introduce the notion of Quillen homotopies.

\begin{definition}
Let $F,\:G\colon\mm\nach\nn$ be left Brown functors. A natural transformation $\tau\colon F\nach G$ is called a \emph{left Quillen homotopy} if the components $\tau_X$ are weak equivalences for all cofibrant objects~$X$.
\end{definition}

\begin{lemma}
Let $F,\:G\colon\mm\nach\nn$ be left Brown functors between combinatorial model categories and let $\tau\colon F\nach G$ be a left Quillen homotopy. Then we obtain a natural isomorphism
$$\LL\!\tau\colon\LL\! F\stackrel{\cong}{\nach} \LL\! G$$
between the induced morphisms of derivators $\LL\! F,\:\LL\! G\colon\DD_{\mm}\nach\DD_{\nn}.$
\end{lemma}

With these preparations we can now turn to monoidal model categories. We use the following definition of a monoidal model category, which is close to the original one in \cite{hovey}.

\begin{definition}
A \emph{monoidal model category} is a model category $\mm$ endowed with a monoidal structure such that the monoidal pairing $\otimes\colon\mm\times\mm\nach\mm$ is a Quillen bifunctor and such that a (and hence any) cofibrant replacement $Q\mathbb{S}\nach \mathbb{S}$ of the monoidal unit has the property that the induced natural transformations $Q\SSS\otimes -\nach\SSS\otimes -$ and $-\otimes Q\SSS\nach -\otimes \SSS$ are Quillen homotopies.
\end{definition}

\begin{theorem}
Let $\mm$ be a combinatorial monoidal model category. The associated derivator $\DD_{\mm}$ inherits canonically the structure of a biclosed monoidal derivator. If the monoidal structure on $\mm$ is symmetric, then this is also the case for the induced structure on $\DD_{\mm}.$
\end{theorem}
\begin{proof}
We only have to put the above results together and care about the unit. The injective model structures on the diagram categories $\mm^J$ have the property that the natural transformations $Q\SSS\otimes -\nach \SSS\otimes -$ and $-\otimes Q\SSS\nach -\otimes \SSS$ are again Quillen homotopies since everything is defined levelwise. Thus, at each stage we can apply the corresponding result of \cite{hovey} to obtain a monoidal structure on $\Ho(\mm^J).$ Moreover, by Corollary \ref{cor_derivedadjunction} these fit together to define a biclosed monoidal structure on $\DD_{\mm}$ since the left Quillen bifunctor $\otimes$ induces a derived adjunction of two variables at the level of derivators. 
\end{proof}

There is a similar result for monoidal left Quillen functors. Recall from \cite{hovey} that a monoidal left Quillen functor is a left Quillen functor which is strong monoidal and satisfies an additional unitality condition. This extra condition ensures that the derived functor will respect the monoidal unit at the level of homotopy categories. We omit the proof that such a monoidal left Quillen functor between combinatorial model categories induces a monoidal morphism of associated derivators. However, after having given the following central examples we will shortly consider the situation of \emph{weakly} monoidal Quillen adjunctions. The first two examples will be taken up again in that context.

\begin{example}
Let~$k$ be a commutative ring and let~$\Ch(k)$ be the category of unbounded chain complexes over~$k$. This category can be equipped with the combinatorial (so-called projective) model structure where the weak equivalences are the quasi-isomorphisms and the fibrations are the surjections (\cite{hovey}). The tensor product of chain complexes endows this category with the structure of a closed monoidal model category. The unit object is given by~$k[0]$ which denotes the chain complex concentrated in degree zero where it takes the value~$k$. Thus, the associated stable \emph{derivator of chain complexes} 
$$\DD_k:=\DD_{\Ch(k)}$$ 
is a closed monoidal derivator.
More generally, let~$C$ be a commutative monoid in~$\Ch(k),$ i.e., let~$C$ be a commutative differential-graded algebra. Then, the category $C-\Mod$ of differential-graded left~$C$-modules inherits a stable, combinatorial model structure (\cite{SchwedeShipley_Alg}). Moreover, forming the tensor product over~$C$ endows $C-\Mod$ with the structure of a closed monoidal model category. We deduce that the associated stable \emph{derivator of differential-graded} $C$-\emph{modules}
$$\DD_C:=\DD_{C-\Mod}$$
is also closed monoidal.
\end{example}

\begin{example}
Let $\sSet$ denote the category of simplicial sets. As a special case of a presheaf category it is a Grothendieck topos and hence, in particular, a presentable category. If we endow it with the homotopy-theoretic Kan model structure (\cite{quillen}, \cite[Chapter 1]{goerss-jardine}) we obtain a Cartesian closed monoidal model category $\sSetKan$. Since this is a combinatorial model category, we obtain a closed monoidal \emph{derivator of simplicial sets}: 
$$\DD_{\sSet}:=\DD_{\sSetKan}$$
But, there is also the Joyal model structure on the category of simplicial sets (see for example \cite{joyal1}, \cite{HTT}, and also \cite{groth_infinity}). This cofibrantly generated model structure~$\sSetJoyal$ is again Cartesian so that we obtain a further closed monoidal derivator, the \emph{derivator of} $\infty-$\emph{categories}:
$$\DD_{\infty-\Cat}:=\DD_{\sSetJoyal}$$
\end{example}

\begin{example}
Let $\Sp^\mathsf{\Sigma}$ be the category of symmetric spectra based on simplicial sets as introduced in \cite{HSS}. This presentable category carries a symmetric monoidal structure given by the smash product $\wedge$ with the sphere spectrum~$\mathbb{S}$ as monoidal unit. It is shown in \cite{HSS} that~$\Sp^\mathsf{\Sigma}$ endowed with the stable model structure is a cofibrantly generated, stable, symmetric monoidal model category in which the unit object is cofibrant. We hence obtain an associated stable, closed monoidal \emph{derivator of spectra}:
$$\DD_{\Sp}:=\DD_{\Sp^\mathsf{\Sigma}}$$
Moreover, let us denote by $E-\Mod$ the category of left $E$-module spectra for a commutative symmetric ring spectrum~$E$. The category $E-\Mod$ can be endowed with the projective model structure by which we mean that the weak equivalences and the fibrations are reflected by the forgetful functor $E-\Mod\nach\Sp^\mathsf{\Sigma}.$ 
This model category is a combinatorial monoidal model category when endowed with the smash product over $E$ and hence gives rise to the stable, closed monoidal derivator of $E$-module spectra:
$$\DD_E:=\DD_{E-\Mod}$$
\end{example}

We will now consider \emph{weakly} monoidal Quillen adjunctions as introduced by Schwede and Shipley in \cite{schwedeshipley_equivalences} and illustrate them by an example. This example will also reveal a technical advantage derivators do have when compared to model categories (cf.\ Corollary \ref{cor_straightmonoidal}). Before we get to that let us give the following result (cf.\ \cite{kelly_doctrinal}). Let us consider an adjunction $(L,R)\colon\mc\rightharpoonup\md$ where both categories~$\mc$ and~$\md$ are monoidal. Moreover, let us assume that we are given a lax monoidal structure on the right adjoint:
$$m\colon RX\otimes RY\nach R(X\otimes Y)\qquad\mbox{and}\qquad u\colon\SSS\nach R\SSS$$
We can now form certain Beck-Chevalley transformed natural transformations associated to~$m$ and~$u.$ In fact, let us define~$m'\colon L(X\otimes Y)\to LX\otimes LY$ and~$u'\colon L\SSS\to\SSS$ by the following pastings respectively:
$$
\xymatrix{
\md\xtwocell[1,1]{}\omit&\mc\xtwocell[1,1]{}\omit\ar[l]_-L & \mc\times\mc\xtwocell[1,1]{}\omit \ar[l]_-\otimes& &
\md\xtwocell[1,1]{}\omit&\mc\xtwocell[1,1]{}\omit\ar[l]_-L & e\xtwocell[1,1]{}\omit \ar[l]_-\SSS& &\\
&\md \ar@/^1.0pc/[lu]^-=\ar[u]^-R & \md\times\md\ar[l]^-\otimes\ar[u]_-{R\times R}& \mc\times\mc\ar@/_1.0pc/[lu]_-= \ar[l]^-{L\times L}&
 &\md \ar@/^1.0pc/[lu]^-=\ar[u]^-R & e\ar[l]^-\SSS\ar[u]_-=& e\ar@/_1.0pc/[lu]_-= \ar[l]^-=&
}
$$
In these pasting diagrams, the additional undecorated natural transformations are again given by the adjunction morphisms. It is now a lengthy formal calculation to show that the pair~$(m',u')$ defines a lax comonoidal structure on~$L$. Similarly, if we start with a lax comonoidal structure on~$L$ given by
$$m'\colon L(X\otimes Y)\nach L(X\otimes Y)\qquad\mbox{and}\qquad u'\colon L\SSS\nach\SSS,$$
we can again form Beck-Chevalley transformed natural transformations in order to obtain a lax monoidal structure on the right adjoint~$R$. 

\begin{lemma}\label{lem_monoidaladjunctions}
Let $\mc$ and $\md$ be monoidal categories and let $(L,R)\colon\mc\rightharpoonup\md$ be an adjunction. The above constructions define a bijection between lax monoidal structures on~$R$ and lax comonoidal structures on~$L$. Moreover, if~$(L,R)$ is an equivalence then we have a bijection between strong monoidal structures on~$L$ and strong monoidal structures on~$R$.
\end{lemma}
\begin{proof}
We have to show that the two constructions are inverse to each other. But this is a special instance of \cite[Lemma 1.19]{groth_derivator}. The second statement for the case of an equivalence of monoidal categories follows immediately from the description of the construction. In fact, in this case the adjunction unit and counit are natural isomorphisms and hence --for example-- the pasting defining~$m'$ is an isomorphism if and only if~$m$ is an isomorphism.
\end{proof}

Let us now recall the following definition of \cite{schwedeshipley_equivalences}.

\begin{definition}
Let $\mm$ and $\nn$ be monoidal model categories. A \emph {weak monoidal Quillen adjunction} $\mm\rightharpoonup\nn$ is a Quillen adjunction $(F,U)$ together with a lax monoidal structure $(m,u)$ on the right adjoint $U$ such that the following two properties are satisfied:\\
{\rm i)} The natural transformation $m'\colon F\circ\otimes\nach\otimes\circ(F\times F)$ which is part of the induced lax comonoidal structure on $F$ is a left Quillen homotopy.\\
{\rm ii)} For any cofibrant replacement $Q\SSS\nach\SSS$ of the monoidal unit $\SSS$ of $\mm$ the map $FQ\SSS\nach F\SSS\stackrel{u'}{\nach}\SSS$ is a weak equivalence.\\
We call such a datum a \emph{weak monoidal Quillen equivalence} if the underlying Quillen adjunction is a Quillen equivalence.
\end{definition}

In the context of combinatorial monoidal model categories one checks that weak monoidal Quillen adjunctions (resp.\ equivalences) can be extended to weak monoidal Quillen adjunctions (resp.\ equivalences) at the level of diagram categories with respect to the injective model structures.

\begin{proposition}
Let $(F,U)\colon\mm\rightharpoonup\nn$ be a weak monoidal Quillen adjunction between combinatorial model categories. Then the left derived morphism $\LL\! F\colon\DD_{\mm}\nach\DD_{\nn}$ carries canonically the structure of a strong monoidal morphism while $\RR\! U\colon\DD_{\nn}\nach\DD_{\mm}$ is canonically lax monoidal. If $(F,U)$ is a weak monoidal Quillen equivalence then both $\LL\! F$ and $\RR\! U$ carry canonically a strong monoidal structure.
\end{proposition}
\begin{proof}
By our assumption the natural transformation $m'\colon F\circ\otimes\nach\otimes\circ(F\times F)$ is a Quillen homotopy. By the additional compatibility assumption of the induced map $u':F\SSS\nach\SSS$ we can use $m'$ and $u'$ in order to obtain a strong comonoidal structure on $\LL\! F\colon\DD_{\mm}\nach\DD_{\nn}.$ Since there is the obvious bijection between strong comonoidal and strong monoidal structures obtained by forming inverse natural transformations, we end up with a strong monoidal structure on~$\LL\! F.$ If~$(F,U)$ is actually a weak monoidal Quillen equivalence, we can apply a variant of Lemma \ref{lem_monoidaladjunctions} for derivators to also construct a strong monoidal structure on $\RR\! U.$ 
\end{proof}

\begin{corollary}\label{cor_straightmonoidal}
Let $\mm,\;\nn$ be combinatorial monoidal model categories which are Quillen equivalent through a zigzag of weakly monoidal Quillen equivalences between combinatorial monoidal model categories. Then we obtain a strongly monoidal equivalence of derivators $\DD_{\mm}\stackrel{\simeq}{\nach}\DD_{\nn}.$
\end{corollary}

As an illustration we want to apply this to the situation described in \cite{shipley_spectradga}. In that paper, Shipley constructs a zigzag of three weak monoidal Quillen equivalences between the category of unbounded chain complexes of abelian groups and the category of $H\mathbb{Z}$-module spectra. To be more specific, the monoidal model for spectra is chosen to be the category of symmetric spectra (\cite{HSS}) and $H\mathbb{Z}$ denotes the integral Eilenberg-MacLane spectrum. The chain of weak monoidal Quillen equivalence passes through the following intermediate model categories
$$H\mathbb{Z}-\Mod\simeq_Q \Sp^\mathsf{\Sigma}(\mathsf{sAb})\simeq_Q\Sp^\mathsf{\Sigma}(\Ch^+)\simeq_Q\Ch.$$
\noindent
Here, $\Ch^+$ is the category of non-negatively graded chain complexes of abelian groups, $\mathsf{sAb}$ is the category of simplicial abelian groups and  $\Sp^\mathsf{\Sigma}(-)$ denotes Hovey's stabilization process by forming symmetric spectra internal to a sufficiently nice model category (\cite{hovey_spectra}). There is a similar such chain of weak monoidal Quillen equivalences if we replace the integers by an arbitrary commutative ground ring $k.$ Since all the four model categories occurring in that chain are combinatorial we can apply the last corollary in order to obtain the following example. 

\begin{example}
For a commutative ring $k$ let us denote by $Hk$ the symmetric Eilenberg-MacLane ring spectrum. Then we have a strong monoidal equivalence of derivators 
$$\DD_k\simeq\DD_{Hk}.$$
\end{example}

\subsection{The bicategory of distributors associated to a monoidal derivator}

In this subsection, $(\EE,\otimes,\SSS)$ will be a (symmetric) monoidal derivator. The aim of this subsection is to show that associated to~$\EE$ there is a (symmetric) monoidal bicategory $\DistE$ of distributors over~$\EE$ (with a reasonably well-behaved notion of a trace). This bicategory together with its canonical action on~$\EE$ encodes a lot of structure. That structure specializes, in particular, to weighted homotopy (co)limits in the case of a closed monoidal derivator and will be taken up again in \cite{groth_enriched} in the context of closed modules over a nice monoidal derivator.

The objects of the bicategory $\DistE$ will just be the small categories (or, in the case of a derivator of type $\Dia$ only the categories lying in $\Dia$). Now, given two such categories~$J$ and~$K,$ for the category of morphisms from~$J$ to~$K$ we set $\DistE(J,K)=\EE(J\times K^{\op}).$ An object $X\in\DistE(J,K)$ will be denoted by $X\colon J\stackrel{\EE}{\nach}K.$ The next aim is to construct a composition functor 
$$\circ_K\colon\DistE(J,K)\times\DistE(K,L)\nach\DistE(J,L).$$
Motivated by the observation that the tensor product of (bi)modules over a classical ring can be considered as a coend construction we proceed as follows. Given $X\colon J\stackrel{\EE}{\nach}K$ and $Y\colon K\stackrel{\EE}{\nach} L$ we define $X\circ_K Y\colon J\stackrel{\EE}{\nach} L$ by:
$$
X\circ_K Y=\int^K X\otimes Y
$$
Thus, the `composition over~$K$' is defined by the following composition:
$$\circ_K\colon\EE(J\times K^{\op})\times \EE(K\times L^{\op})\stackrel{\otimes}{\nach}\EE(J\times K^{\op}\times K\times L^{\op})\stackrel{\int^K}{\nach}\EE(J\times L^{\op})$$
The associativity constraint of the composition is obtained as follows. Let us assume we are given three distributors $X\colon J\stackrel{\EE}{\nach}K,$ $Y\colon K\stackrel{\EE}{\nach}L,$ and $Z\colon L \stackrel{\EE}{\nach}M.$ Then an isomorphism
$$
(X\circ_K Y)\circ_L Z\cong X\circ_K (Y\circ_L Z)
$$ 
in $\DistE(J,M)$ is obtained by the following pasting of natural isomorphisms:
\begin{eqnarray*}
(X\circ_K Y)\circ_L Z&=&\int^L\big(\int^K X\otimes Y\big)\otimes Z\\
&\cong&\int^L\int^K\big( (X\otimes Y)\otimes Z\big)\\
&\cong&\int^K\int^L \big(X\otimes (Y\otimes Z) \big)\\
&\cong&\int^KX\otimes(\int^L Y\otimes Z)\\
&=&X\circ_K(Y\circ_L Z)
\end{eqnarray*}
In this chain, the first and the last isomorphism are using that the monoidal structure preserves homotopy coends separately in each variable (Lemma \ref{lemma_cocontcoend}). The isomorphism in the middle is obtained by a combination of the Fubini-like Theorem (Lemma \ref{lemma_Fubini}) with the associativity constraint of~$\EE.$

The bicategory $\DistE$ also carries a monoidal structure. On objects, this monoidal structure is given by the product of small categories, while on morphism categories it is essentially given by the bimorphism~$\otimes$. More precisely, the functor $$\otimes\colon\DistE(J_1,K_1)\times\DistE(J_2,K_2)\nach\DistE(J_1\times J_2,K_1\times K_2)$$ 
is given by:
$$
\xymatrix{
\EE(J_1\times K_1^{\op})\times\EE(J_2\times K_2^{\op})\ar[r]^-\otimes& \EE(J_1\times K_1^{\op}\times J_2\times K_2^{\op})\ar[r]^-t&\EE(J_1\times J_2\times (K_1\times K_2)^{\op})
}
$$
One easily constructs a symmetry constraint for the case of a symmetric monoidal derivator~$\EE.$ It can be shown that one gets the following result.

\begin{theorem}
If $\EE$ is a (symmetric) monoidal derivator then there is a (symmetric) monoidal bicategory $\DistE$ of distributors in~$\EE$ with objects the small categories and morphism categories given by~$\DistE(J,K)=\EE(J\times K^{\op}).$
\end{theorem}

If the monoidal structure on~$\EE$ happens to be symmetric, then $\DistE$ allows for a well-behaved notion of a trace associated to an `endo-distributor'. For a small category~$J$, let us define the \emph{trace} by:
$$tr_J=tr\colon\DistE(J,J)\nach\EE(e)\colon X\mapsto tr(X)=\int^J X$$
This trace is well-behaved in the sense that we have the following two formulas. For convenience, in the proposition, we denote the composition functors $\circ_K$ simply by juxtaposition.

\begin{proposition}
Let~$\EE$ be a symmetric monoidal derivator and let us consider distributors $X\colon J\stackrel{\EE}{\nach}K,$ $Y\colon K\stackrel{\EE}{\nach}J,$ and $Z_i\colon J_i\stackrel{\EE}{\nach}J_i,\:i=1,2.$  Then there are natural isomorphisms:
$$
tr(XY)\cong tr(YX)\qquad\mbox{and}\qquad tr(Z_1\otimes Z_2)\cong tr(Z_1)\otimes tr(Z_2)
$$
\end{proposition}
\begin{proof}
These formulas follow from the Fubini-type Theorem for homotopy coends (Lemma \ref{lemma_Fubini}).
\end{proof}
\noindent
Moreover, these trace functors satisfy certain coherence conditions which are encoded by the notion of a `shadow functor' (\cite{shulmanponto_shadows}). This bicategory will be studied in more detail in \cite{grothpontoshulman_distributors}.

Let us now turn to the case of a \emph{closed} monoidal derivator. In that case the structure given by the bicategory of distributors~$\DistE$ contains, in particular, the \emph{weighted homotopy (co)limit functors.} In order to see this let us establish the following easy but very convenient lemma.

\begin{lemma}\label{lemma_divisible}
Let~$\mc,\:\md,$ and~$\me$ be categories and let $\otimes\colon\mc\times\md\nach\me$ be a functor which is divisible on both sides. If $L\colon\me\nach\me'$ is a left adjoint functor then the composition $L\circ\otimes\colon\mc\times\md\nach\me'$ is also divisible on both sides.
\end{lemma} 

\begin{corollary}
If~$\EE$ is a closed monoidal derivator then the bicategory~$\DistE$ of distributors in~$\EE$ has a composition law which is divisible on both sides.
\end{corollary}
\begin{proof}
It suffices to observe that the composition law~$\circ_K\colon\DistE(J,K)\times\DistE(K,L)\nach\DistE(J,L)$ is defined by the following composition:
$$\EE(J\times K^{\op})\times\EE(K\times L^{\op})\stackrel{\otimes}{\nach} \EE(J\times K^{\op}\times K\times L^{\op})\stackrel{(t,s)^\ast}{\nach} \EE(J\times Ar(K)^{\tw}\times L^{\op})\stackrel{p_!}{\nach}\EE(J\times L^{\op})$$
By our assumption on $\EE,$ the first functor in this composition is divisible on both sides. Moreover, the other two functors are left adjoints so that we can conclude by Lemma \ref{lemma_divisible}.
\end{proof}

Thus, given a closed monoidal derivator~$\EE$ then $\circ_K\colon\DistE(J,K)\times\DistE(K,L)\nach\DistE(J,L)$ is divisible on both sides by certain functors
$$\DistE(J,K)^{\op}\times\DistE(J,L)\nach\DistE(K,L)\qquad\mbox{and}\qquad \DistE(K,L)^{\op}\times\DistE(J,L)\nach\DistE(J,K)$$
Specializing the composition law to the case of $J=L=e$ and the first of the above division functors to the case $K=L=e$ we obtain functors
$$\EE(K^{\op})\times\EE(K)\nach\EE(e)\qquad\mbox{and}\qquad\EE(J)^{\op}\times\EE(J)\nach\EE(e)$$
which give us \emph{weighted homotopy colimit} and \emph{weighted homotopy limit functors} respectively. The second division functor gives rise to some sort of enrichment of the derivator~$\EE.$ We will take up these issues again in the sequel to this paper (cf.\ \cite{groth_enriched}).

\section{Additive derivators}\label{section_additive}

\subsection{Additive derivators and the canonical pretriangulated structures}

For a derivator $\DD$ and a category $J$ it is immediate that $\DD(J)$ has initial and final objects as well as finite coproducts and finite products (cf.\ Subsection 1.1 of \cite{groth_derivator}). A pointed derivator is a derivator such that every initial object of the underlying category $\DD(e)$ is also final. It follows then that all values $\DD(J)$ are pointed. For additive derivators, it similarly suffices to impose an additivity assumption on the underlying category. For us the notion of an additive category does \emph{not} include an enrichment in abelian groups. The additional \emph{structure} given by the enrichment in abelian groups can be uniquely reconstructed using the exactness \emph{properties} of an additive category. Thus, the category $\DD(e)$ is assumed to be pointed and the canonical map from the coproduct of two objects to the product of them is to be an isomorphism. Moreover, for every object there is a self-map which `behaves as an additive inverse of the identity'. For a precise formulation of this axiom, compare to Definition~$8.2.8$ of \cite{schapira}. Alternatively, one can demand the shear map 
$${1\;\;1 \choose 0\;\;1}\colon X\sqcup X\nach X\times X$$
to be an isomorphism for each object $X$. 

\begin{definition}\label{def_add}
A derivator $\DD$ is \emph{additive} if the underlying category $\DD(e)$ is additive.
\end{definition}

\begin{proposition}
If a derivator $\DD$ is additive, then all categories $\DD(J)$ are additive and for any functor $u\colon J\nach K$ the induced functors $u^\ast,\;u_!,$ and $u_\ast$ are additive.
\end{proposition}
\begin{proof}
Let us assume $\DD$ to be additive and let us consider an arbitrary category $J$. We already know that $\DD(J)$ is pointed. Since isomorphisms in $\DD(J)$ can be tested pointwise and since the evaluation functors have adjoints on both sides it is easy to see that finite coproducts and finite products in $\DD(J)$ are canonically isomorphic. Similarly, let $X\in\DD(J)$ be an arbitrary object and let us consider the shear map ${1\;\;1 \choose 0\;\;1}\colon X\sqcup X\nach X\times X.$ This map is an isomorphism if and only if this is the case when evaluated at all objects $j\in J$. But $j^\ast{1\;\;1 \choose 0\;\;1}$ can be canonically identified with the shear map of $j^\ast X\in\DD(e)$ which is an isomorphism by assumption. Finally, given a functor $u\colon J\nach K,$ the induced functors $u^\ast,\;u_!,$ and $u_\ast$ are all additive since each of them has an adjoint on at least one side.
\end{proof}

In contrast to the above definition, let us call a prederivator additive if all values and all precomposition functors are additive. Let us recall from \cite[Theorem 1.31]{groth_derivator} that given a derivator~$\DD$ and a small category~$M$ then the prederivator~$\DD^M\colon J\mapsto\DD(J\times M)$ is again a derivator. A combination of this together with the last proposition gives the fourth example.

\begin{example}\label{ex_additive}
{\rm i)}  Let~$R$ be a ring and let $\Ch_{\geq 0}(R)$ denote the category of non-negative chain complexes of left~$R$-modules. It is shown in \cite[Section 7]{dwyerspalinski} that~$\Ch_{\geq 0}(R)$ can be endowed with a cofibrantly-generated model structure with quasi-isomorphisms as weak equivalences. Hence,
$$\DD_R^{\geq 0}:=\DD_{\Ch_{\geq 0}(R)}\colon\Cat^{\op}\to\CAT\colon J\mapsto \Ho(\Ch_{\geq 0}(R)^J)=D_{\geq 0}(R-\Mod^J)$$
is an example of an additive derivator which is \emph{not} stable. Here,~$D_{\geq 0}(-)$ denotes the formation of the non-negative derived category of an abelian category.\\
{\rm ii)} Let $\DD$ be a stable derivator. Then we showed in Section 4 of \cite{groth_derivator} that $\DD$ is also an additive derivator. So, this applies, in particular, to derivators associated to stable (combinatorial) model categories or exact categories in the sense of Quillen (\cite{quillen_ktheory}).\\
{\rm iii)} The prederivator represented by a category is additive if and only if the representing category is additive.\\
{\rm iv)} A derivator~$\DD$ is additive if and only if the derivator~$\DD^M$ is additive for each small category~$M.$\\
{\rm v)} A derivator~$\DD$ is additive if and only if~$\DD^{\op}$ is additive.
\end{example}

In \cite[Section 4]{groth_derivator} we showed that the values of a \emph{stable} derivator (see \cite[Def.\ 4.1]{groth_derivator}) can be canonically endowed with the structure of a triangulated category. Moreover, the precomposition and the homotopy Kan extension functors can be canonically turned into exact functors with respect to these triangulations. A careful analysis of the corresponding proofs will show that we obtain weaker versions of such results in the context of a strong, additive derivator (see Theorems~\ref{thm_leftright},~\ref{thm_pretriang} and Corollary~\ref{cor_can}). For convenience let us recall the definition of a right triangulated category (see \cite{kellervossieck, beligiannis_left}).

\begin{definition}
Let~$\ma$ be an additive category with an additive functor~$\Sigma\colon\ma\nach\ma$ and a class of so-called \emph{distinguished right triangles} $X\nach Y\nach Z\nach \Sigma X.$ The pair consisting of~$\Sigma$ and the class of distinguished right triangles determines a \emph{right triangulated structure} on~$\ma$ if the following four axioms are satisfied. In this case, the triple consisting of the category, the endofunctor, and the class of distinguished triangles is called a \emph{right triangulated category}.\\
{\rm (RT1)} For every $X\in\ma,$ the right triangle $0\to X\stackrel{\id}{\to}X\to 0$ is distinguished. Every morphism in~$\ma$ occurs as the first morphism in a distinguished right triangle and the class of distinguished right triangles is replete, i.e., is closed under isomorphisms.\\
{\rm (RT2)} If $X\stackrel{f}{\nach} Y\stackrel{g}{\nach} Z\stackrel{h}{\nach} \Sigma X$ is a distinguished right triangle then so is the rotated right triangle $Y\stackrel{g}{\nach} Z\stackrel{h}{\nach} \Sigma X\stackrel{-f}{\nach}\Sigma Y$ is.\\
{\rm (RT3)} Given two distinguished right triangles and a commutative solid arrow diagram
$$\xymatrix{
X\ar[r]\ar[d]^u& Y\ar[r]\ar[d]^v & Z\ar[r]\ar@{-->}[d]^w & \Sigma X\ar[d]^{\Sigma u}\\
X'\ar[r] & Y'\ar[r] & Z'\ar[r] & \Sigma X'
}
$$
there exists a dashed arrow $w\colon Z\nach Z'$ as indicated such that the extended diagram commutes.\\
{\rm (RT4)} For every pair of composable arrows $f_3\colon X\stackrel{f_1}{\nach} Y\stackrel{f_2}{\nach} Z$ there is a commutative diagram
$$\xymatrix{
X\ar[r]^{f_1}\ar@{=}[d]& Y \ar[r]^{g_1}\ar[d]_{f_2}& C_1\ar[r]^{h_1}\ar[d]& \Sigma X\ar@{=}[d]\\
X\ar[r]_{f_3}& Z\ar[d]_{g_2}\ar[r]_{g_3}& C_3\ar[r]_{h_3}\ar[d]& \Sigma X\ar[d]^{\Sigma f_1}\\
&C_2\ar[d]_{h_2}\ar@{=}[r]&C_2\ar[d]^{\Sigma g_1\circ h_2}\ar[r]_{h_2}&\Sigma Y\\
&\Sigma Y\ar[r]_{\Sigma g_1}& \Sigma C_1& 
}
$$
in which the rows and columns are distinguished right triangles.
\end{definition}

Recall from \cite[Subsection 3.3]{groth_derivator} that given a pointed derivator we have an adjunction $(\Sigma,\Omega)\colon\DD(J)\rightharpoonup\DD(J).$ In \cite[Subsection 4.1]{groth_derivator} we established the additivity of a stable derivator. A key step in that context was to to show that loop objects in a stable derivator are \emph{group} objects as opposed to only monoids. But that proof did not use the stability in an essential way. In fact, we have the following result. For the concatenation of loops and inversion of loops see \cite[Subsection 4.1]{groth_derivator}.

\begin{proposition}\label{prop_loopgroup}
Let~$\DD$ be an additive derivator and let $X\in\DD(J)$ for some small category~$J$. Then the concatenation of loops $\ast\colon\Omega X\oplus\Omega X\to\Omega X$ and the inversion of loops~$\sigma^\ast\colon\Omega X\to\Omega X$ turn~$\Omega X$ into a group object of~$\DD(J).$ Moreover, given an object~$U\in\DD(J)$ and morphisms $f,\:g\colon U\to \Omega X$ then we have:
$$f+g\quad=\quad f\ast g\qquad\mbox{and}\qquad -f\quad = \quad\sigma^\ast f$$
\end{proposition}
\begin{proof}
Exactly the same proof as in the stable case does the job.
\end{proof}

This result is slightly nicer than the corresponding one in the stable case: given an additive derivator we already had \emph{both} an addition and a multiplication by -1 on the set of morphisms from~$U$ to~$\Omega X$ and both of them can be interpreted geometrically by some `loop manipulation'. Of course, there is a dual result for maps out of objects of the form~$\Sigma Y$.

Now, given an additive derivator then the suspension functor~$\Sigma\colon\DD(J)\to\DD(J)$ is additive since it is a left adjoint. Using precisely the same reasoning as in \cite[Subsection 4.2]{groth_derivator} we define a replete class of distinguished right triangles in~$\DD(J)$.

\begin{theorem}\label{thm_leftright}
Let~$\DD$ be a strong, additive derivator and let~$J$ be a small category. Then the pair consisting of~$\Sigma\colon\DD(J)\to\DD(J)$ and the above class of distinguished right triangles defines a right triangulated structure on~$\DD(J).$ Dually, the pair consisting of~$\Omega\colon\DD(J)\to\DD(J)$ and the dually defined class of distinguished left triangles turns~$\DD(J)$ into a left triangulated category.
\end{theorem}
\begin{proof}
We only mention the necessary adaptations of the proof in the stable case to this more general context. By Example \ref{ex_additive} we can restrict attention to the case~$J=e.$ Moreover, we only discuss the right triangulation.\\
{\rm (LT1):} The repleteness follows by definition. The strongness assumption on~$\DD$ implies that every morphism in~$\DD(e)$ can be extended to a distinguished right triangle. Finally, let $1\colon e\to[1]=(0<1)$ be the functor classifying the object~1. Then, \cite[Prop.\ 3.6]{groth_derivator} guarantees that $1_!\colon\DD(e)\to\DD([1])$ is an `extension by zero functor'. Given an object~$X\in\DD(e)$ \cite[Prop.\ 3.13]{groth_derivator} implies that the distinguished right triangle associated to~$1_!(X)$ looks like~$0\to X\stackrel{\cong}{\to} X'\to 0.$ The repleteness implies that $0\to X\stackrel{=}{\to}X\to 0$ is also distinguished.\\
{\rm (LT2):} The proof is the same as in the stable case with the difference that we have to invoke Proposition~\ref{prop_loopgroup} this time.\\
{\rm (LT3), (LT4):} The corresponding parts of the proof in the stable case work without any change.
\end{proof}

There is even a stronger result. Given an additive category together with both a left and a right triangulated structure one might ask them to be compatible in the following sense (cf.\ \cite{beligiannis_homotopy}). Recall that given an adjunction we always denote the adjunction unit by~$\eta$ and the adjunction counit by~$\epsilon.$

\begin{definition}\label{def_pretriang}
Let $\ma$ be an additive category and let~$(\Sigma,\Omega)\colon\ma\rightharpoonup\ma$ be an adjunction such that~$\Sigma$ and~$\Omega$ are part of a right and left triangulation on~$\ma$ respectively. This quadruple is called a \emph{pretriangulation on}~$\ma$ if the following properties are satisfied:\\
{\rm (PT1):} Let us be given a right triangle and a left triangle in~$\ma$ as indicated in the next diagram. If we have morphisms $\alpha$ and~$\beta$ such that the square on the left is commutative then there is a morphism~$\gamma\colon Z\to Y'$ such that the entire diagram commutes:
$$
\xymatrix{
X\ar[r]^-f\ar[d]_\alpha&Y\ar[r]^-g\ar[d]_\beta&Z\ar[r]^-h\ar@{-->}[d]^\gamma&\Sigma X\ar[d]^{\epsilon\circ\Sigma\alpha}\\
\Omega Z'\ar[r]_-{f'}&X'\ar[r]_-{g'}&Y'\ar[r]_-{h'}&Z'
}
$$
{\rm (PT2):} Let us be given a right triangle and a left triangle in~$\ma$ as indicated in the next diagram. If we have morphisms $\alpha$ and~$\beta$ such that the square on the right is commutative then there is a morphism~$\gamma\colon Y\to X'$ such that the entire diagram commutes:
$$
\xymatrix{
X\ar[r]^-f\ar[d]_{\Omega\alpha\circ\eta}&Y\ar[r]^-g\ar@{-->}[d]_\gamma&Z\ar[r]^-h\ar[d]^\beta&\Sigma X\ar[d]^\alpha\\
\Omega Z'\ar[r]_-{f'}&X'\ar[r]_-{g'}&Y'\ar[r]_-{h'}&Z'
}
$$
A \emph{pretriangulated category} is an additive category together with a pretriangulation on it. 
\end{definition}

We have the following nice theorem about the values of a strong, additive derivator.

\begin{theorem}\label{thm_pretriang}
Let~$\DD$ be a strong, additive derivator and let~$J$ be a small category. Then the adjunction $(\Sigma,\Omega)\colon\DD(J)\rightharpoonup\DD(J)$ together with the right and the left triangulated structure on~$\DD(J)$ guaranteed by Theorem~\ref{thm_leftright} turn $\DD(J)$ into a pretriangulated category.
\end{theorem}
\begin{proof}
By Example~\ref{ex_additive} it suffices to consider the case of~$J=e.$ Moreover, by duality it suffices to establish one of the compatibility properties. We will give a proof of {\rm (PT1)}.

Before we can give the actual proof we have to recall some details about the construction of the right triangulation. For this purpose, let us consider the category~$K$ which is given by the following poset:
$$
\xymatrix{
(0,0)\ar[d]\ar[r]&(1,0)\ar[r]&(2,0)\\
(0,1)&&
}
$$
Moreover, let~$i_0\colon[1]\to K$ be the functor classifying the left horizontal arrow and let $i_1\colon K\to[2]\times[1]$ be the inclusion. Since~$i_0$ is a sieve, Proposition 3.6 of \cite{groth_derivator} can be applied to deduce that ${i_0}_\ast\colon\DD([e])\to\DD(K)$ is an `extension by zero functor'. The right triangulation on~$\DD(e)$ is obtained from the functor~$T_\Sigma\colon\DD([1])\stackrel{{i_0}_\ast}{\to}\DD(K)\stackrel{{i_1}_!}{\to}\DD([2]\times[1]).$ More precisely, given an object $f\in\DD([1])$ with underlying diagram $f\colon X\to Y$ in~$\DD(e)$ then~$T_\Sigma(f)_{0,1}$ and~$T_\Sigma(f)_{2,0}$ both vanish. Since the compound square~$d_1T_\Sigma(f)$ is coCartesian we obtain a canonical isomorphism $T_\Sigma(f)_{2,1}\cong\Sigma X.$ If we now restrict the objects~$T_\Sigma(f)$ to objects in~$\DD([3])$ in the obvious way, then pass to underlying diagrams in~$\DD(e),$ and form the replete closure, then we obtain the right triangulation on~$\DD(e)$. The left triangulation is obtained by a dual construction. The first step of that construction consists of defining a functor~$T_\Omega\colon\DD([1])\to\DD([2]\times[1])$ which is obtained by first applying an `extension by zero functor' and then forming a homotopy right Kan extensions which basically forms two composable homotopy pullbacks.

Now, let us consider the situation of {\rm (PT1)}. We can assume that the first row is the underlying diagram of~$T_\Sigma(f)$ and that the second one is the underlying diagram of~$T_\Omega(h').$ Let us not distinguish notationally between $f\in\DD([1])$ and its underlying diagram $f\colon X\to Y$ in~$\DD(e)$ and similarly for other morphisms. We will construct the morphism~$\gamma$ in two steps. First, by the strongness of our derivator~$\DD$ we can find a morphism $\phi\colon f\to f'$ in~$\DD([1])$ such that the underlying diagram of~$\phi$ is precisely~$(\alpha,\beta).$ From this we get a morphism~$T_\Sigma(\phi)\colon T_\Sigma(f)\to T_\Sigma(f')$ and it is easy to verify that under our identifications the morphism~$T_\Sigma(\phi)_{2,1}$ is just~$\Sigma\alpha.$

For the second step observe that we have an isomorphism $T_\Sigma(f')\!\mid_{[1]}\cong T_\Omega(h')\!\mid_{[1]}$ which induces by homotopy right Kan extension along~$i_0$ a further isomorphism $T_\Sigma(f')\!\mid_K\cong T_\Omega(h')\!\mid_K.$ Now combining the adjunction $({i_1}_!,{i_1}^\ast)$ together with the canonical isomorphism ${i_1}_!(T_\Sigma(f')\!\mid_K)\cong T_\Sigma(f')$ (which we have since Prop.\ 1.26 of \cite{groth_derivator} guarantees the fully-faithfulness of~${i_1}_!$) we obtain a morphism:
$$T_\Sigma(f')\stackrel{\cong}{\to}{i_1}_!(T_\Sigma(f')\!\mid_K)\stackrel{\cong}{\to} {i_1}_!(T_\Sigma(h')\!\mid_K)\stackrel{\epsilon}{\to}T_\Omega(h')$$
This morphism evaluated at $(2,1)$ can be identified with the adjunction unit $\epsilon\colon\Sigma\Omega Z'\to Z'.$ Thus, if we define $\gamma\colon Z\to Y'$ to be the morphism $T_\Sigma(f)\to T_\Sigma(f')\to T_\Omega(h')$ evaluated at $(1,1)$ then we can conclude the proof.
\end{proof}

We will refer to the left triangulated, right triangulated, and pretriangulated structures of the last two theorems as the \emph{canonical structures}. In order to make a precise statement that the pretriangulations on the various values of an additive derivator are compatible let us recall the following terminology from \cite{beligiannisreiten}. A \emph{right exact morphism} between right triangulated categories is a pair consisting of an additive functor~$F\colon\ma\to\ma'$ and a natural isomorphism $F\circ\Sigma\cong\Sigma\circ F$ which together send distinguished right triangles to distinguished right triangles. There is the obvious dual notion of a \emph{left exact morphism} of left triangulated categories. Moreover, an \emph{exact morphism between pretriangulated categories} is an additive functor which is endowed both with a right exact and a left exact structure. Finally, a \emph{morphism of pretriangulated categories}~$\ma$ and~$\ma'$ is an adjunction $(L,R)\colon\ma\rightharpoonup\ma'$ such that the left adjoint is right exact and the right adjoint is left exact in the above sense. In the context of pretriangulated categories the exactness assumptions on~$L$ and~$R$ are \emph{not} formal consequences of the adjointness since some choices where made earlier in the construction of the pretriangulations. However, for derivators the corresponding statement is true.

\begin{proposition}
Let~$\DD$ and~$\DD'$ be strong, additive derivators and let~$F\colon\DD\to\DD'$ be a morphism which preserves homotopy colimits. Then $F_J\colon\DD(J)\to\DD'(J)$ can be canonically turned into a right exact functor with respect to the canonical right triangulated structures on~$\DD(J)$ and~$\DD(J)'$.
\end{proposition}
\begin{proof}
Since~$F$ preserves homotopy colimits it preserves in particular zero objects and hence homotopy right Kan extensions along inclusions of sieves as these are `extension by zero functors' (\cite[Proposition 3.6]{groth_derivator}). Thus,~$F$ is compatible with $T_\Sigma$ up to natural isomorphism. From this one easily constructs the right exact structure on~$F_J\colon\DD(J)\to\DD'(J).$
\end{proof}

\begin{corollary}\label{cor_can}
Let $(L,R)\colon\DD\rightharpoonup\DD'$ be an adjunction between strong, additive derivators, then we obtain canonically morphisms of pretriangulated categories $(L_J,R_J)\colon\DD(J)\rightharpoonup\DD'(J).$ Moreover, given a functor $u\colon J\to K$ between small categories, we obtain morphisms of pretriangulated categories $(u_!,u^\ast)\colon\DD(J)\rightharpoonup\DD(K)$ and $(u^\ast,u_\ast)\colon\DD(K)\rightharpoonup\DD(J).$ In particular, $u^\ast\colon\DD(K)\to\DD(J)$ is naturally an exact morphism of pretriangulated categories.
\end{corollary}
\begin{proof}
The claim about adjunctions follows immediately from the last proposition since left adjoints preserve homotopy colimits and dually for right adjoints. If we specialize this to the underlying functors of the adjunctions $(u_!,u^\ast)\colon\DD^J\rightharpoonup\DD^K$ and $(u^\ast,u_\ast)\colon\DD^K\rightharpoonup\DD^J$ (see \cite[Example 2.13]{groth_derivator}) we obtain the remaining claims.
\end{proof}

\subsection{The center of an additive derivator and linear structures}

\begin{definition}\label{def_center}
Let $\DD$ be a prederivator. The \emph{center} $\Z(\DD)$ \emph{of} $\DD$ is the set of natural transformations $$\Z(\DD)=\nat(\id_{\DD},\id_{\DD}).$$
\end{definition}
\noindent
Thus, an element of $\Z(\DD)$ is a natural transformation $\tau\colon\id_{\DD}\nach\id_{\DD}$, i.e., a family of natural transformations $\tau_J\colon \id_{\DD(J)}\nach \id_{\DD(J)}$ which behave well with the precomposition functors $u^\ast.$ The composition of natural transformations endows $\Z(\DD)$ with the structure of a (commutative) monoid.

\begin{lemma}
The center~$\Z(\DD)$ of an additive prederivator~$\DD$ is a commutative ring. 
\end{lemma}
\begin{proof}
The multiplication on $\Z(\DD)$ is given by composition. For two elements $\tau,\:\sigma\in \Z(\DD),$ a category $K$ and an element $X\in\DD(K),$ by naturality we have the following commutative diagram:
$$
\xymatrix{
X\ar[r]^{(\tau_K)_X}\ar[d]_{(\sigma_K)_X}&X\ar[d]^{(\sigma_K)_X}\\
X\ar[r]_{(\tau_K)_X}&X
}
$$
Thus we have $\sigma\tau=\tau\sigma,$ i.e., the multiplication is commutative. Since the precomposition functors $u^\ast$ are additive, the sum $\tau + \sigma$ of two elements $\tau,\;\sigma\in\Z(\DD)$ again lies in the center. Finally, the biadditivity of the composition in the additive situation concludes the proof.
\end{proof}

As in classical category theory, this commutative ring~$\Z(\DD)$ can be used to endow an additive prederivator with~$k$-linear structures as follows.

\begin{definition}
Let~$\DD$ be an additive prederivator and let~$k$ be a commutative ring. A~$k$-\emph{linear structure on}~$\DD$ is a ring homomorphism
$$\sigma\colon k\nach \Z(\DD).$$
\noindent
A pair~$(\DD,\sigma)$ consisting of an additive prederivator~$\DD$ and a $k$-linear structure~$\sigma$ on~$\DD$ is a $k$-\emph{linear prederivator}.
\end{definition}

As emphasized in the definition, $k$-linearity of an additive prederivator is additional \emph{structure} (contrary to the additivity which is a \emph{property}). Nevertheless, we will drop~$\sigma$ from notation and speak of a~$k$-linear additive prederivator~$\DD$. Every additive prederivator is canonically endowed with a $\mathbb{Z}$-linear structure. 

Now, let $\DD$ be an additive prederivator. Evaluation at a category~$J$ induces a ring homomorphism $\Z(\DD)\nach \Z(\DD(J)),$ where~$\Z(\DD(J))$ denotes the usual center of the additive category~$\DD(J)$, i.e., the commutative ring of natural transformations $\id_{\DD(J)}\nach\id_{\DD(J)}.$ Thus, a~$k$-linear structure on an additive prederivator induces~$k$-linear structures on all its values. Moreover, these $k$-linear structures are preserved by the precomposition functors. Recall for example from \cite{schapira} that for a morphism $f\colon X\nach Y$ in~$\DD(K)$ and a ring element $s\in k$ the morphism $sf\colon X\nach Y$ is given by the diagonal in the following commutative diagram:
$$
\xymatrix{
X\ar[r]^s\ar[d]_f&X\ar[d]^f\\
Y\ar[r]_s&Y
}
$$
Here, we simplified notation by writing $s$ for $(\sigma(s)_K)_X$ and $(\sigma(s)_K)_Y$ respectively. Now, since~$\sigma(s)$ lies in~$\Z(\DD)$ we have an equality of natural transformations $u^\ast s=su^\ast\colon u^\ast\nach u^\ast$ for an arbitrary functor $u\colon J\nach K.$ For a morphism $f\colon X\nach Y$ in $\DD(K)$ this equality implies $u^\ast(sf)=su^\ast(f),$ i.e., the $k$-linearity of $u^\ast$. Conversely, $k$-linear structures on the values of an additive prederivator such that the precomposition functors are $k$-linear give a $k$-linear structure on the additive prederivator. This gives the first part of the following proposition.

\begin{proposition}
Let $\DD$ be an additive prederivator. A $k$-linear structure on~$\DD$ is equivalently given by a $k$-linear structure on $\DD(J)$ for each category~$J$ such that the precomposition functors are $k$-linear. If~$\DD$ is~$k$-linear derivator, then also the homotopy Kan extension functors are $k$-linear.
\end{proposition}
\begin{proof}
It remains to give a proof of the second statement and, by duality, it suffices to treat the case of homotopy left Kan extensions. Let $X,Y$ be objects of $\DD(J)$ and let $s\in k$. Let us consider the following commutative diagram in which the horizontal isomorphisms are the adjunction isomorphisms:
$$
\xymatrix{
\hhom_{\DD(K)}(u_!X,u_!Y)\ar[r]^\cong\ar[d]_{s_\ast}&\hhom_{\DD(J)}(X,u^\ast u_!Y)\ar[d]^{(u^\ast(s))_\ast}\\
\hhom_{\DD(K)}(u_!X,u_!Y)&\hhom_{\DD(J)}(X,u^\ast u_!Y)\ar[l]^\cong
}
$$
The vertical map on the left sends $u_!(f)\colon u_!X\nach u_!Y$ to~$su_!(f).$ So let us calculate the image of~$u_!(f)$ under the composition of the three maps. Let us remark first that $(u^\ast(s))_\ast=s_\ast$ since~$u^\ast$ is~$k$-linear (which shows, in particular, that the adjunction isomorphisms are~$k$-linear). Thus, the image of~$u_!(f)$ under the composition of the first two maps is the composition of the following commutative diagram:
$$
\xymatrix{
X\ar[r]^-\eta\ar[d]_s&u^\ast u_!X\ar[d]_s\ar[r]^-f&u^\ast u_!Y\ar[d]^s\\
X\ar[r]_-\eta&u^\ast u_!X\ar[r]_-f&u^\ast u_!Y
}
$$
But, using the triangular identities, this composition is sent by the second adjunction isomorphism to $u_!(f)u_!(s)=u_!(sf).$ Hence, we obtain the intended relation $u_!(sf)=su_!(f)$ expressing the $k$-linearity of~$u_!.$
\end{proof}

We finish by giving the notion of $k$-linear morphisms of $k$-linear prederivators. Let us note that an additive morphism $F\colon \DD\nach\DD'$ of additive prederivators induces ring maps $F_\ast\colon \Z(\DD)\nach\nat(F,F)$ and $F^\ast\colon \Z(\DD')\nach\nat(F,F).$

\begin{definition}
Let $\DD$ and $\DD'$ be $k$-linear prederivators with respective $k$-linear structures~$\sigma$ and~$\sigma'.$ An additive morphism $F\colon \DD\nach \DD'$ is $k$-\emph{linear} if $F_\ast\circ\sigma=F^\ast\circ\sigma'\colon k\nach \nat(F,F).$ 
\end{definition}

It is easy to see that an additive morphism $F\colon\DD\nach\DD'$ of $k$-linear prederivators is $k$-linear if and only if all components $F_K\colon\DD(K)\nach\DD'(K)$ are $k$-linear functors. Thus, an additive morphism of additive prederivators is automatically~$\mathbb{Z}$-linear. Other examples of linear structures will be given in the next subsection and in the sequel \cite{groth_enriched}.

In the case of a stable derivator there is a $\mathbb{Z}$-graded variant of the center. Let us recall first that a stable derivators is canonically a `$\mathbb{Z}$-graded derivator'. Given a stable derivator~$\DD$ recall from Section~4 of \cite{groth_derivator} that the suspension functor $\Sigma\colon\DD(J)\nach\DD(J)$ is defined as a certain composition of precomposition and homotopy Kan extension functors. Since these functors assemble to morphisms of derivators (see \cite[Ex.\ 2.13]{groth_derivator}) we obtain a \emph{suspension morphism} $\Sigma\colon\DD\to\DD.$ In the notation of loc.\ cit.\ we have:
$$
\Sigma\colon\DD\stackrel{(0,0)_\ast}{\nach}\DD^{\push}\stackrel{\ipush_!}{\nach}\DD^\square\stackrel{(1,1)^\ast}{\nach}\DD
$$
This already applies in the pointed case but in the stable case~$\Sigma\colon\DD\to\DD$ is a self-equivalence. Using this fact we can consider the values of a stable derivator as $\mathbb{Z}$-graded categories in the following way. For a category~$J$ and two objects $X,Y\in\DD(J),$ the graded abelian groups $\hhom_{\DD(J)}(X,Y)_\bullet$ and $\hhom_{\DD(J)}(X,Y)^\bullet$ are defined to be
$$\hhom_{\DD(J)}(X,Y)_n\quad=\quad\hhom_{\DD(J)}(X,Y)^{-n}\quad=\quad\hhom_{\DD(J)}(\Sigma^n X,Y),\qquad n\in\mathbb{Z}.$$ For a functor $u\colon J\to K$ the induced functors $u^\ast,\;u_!,$ and $u_\ast$ are graded since they are even exact with respect to the canonical triangulated structures (\cite[Cor.\ 4.19]{groth_derivator}). Moreover, this $\mathbb{Z}$-grading of stable derivators is canonical in that it is respected by exact morphisms. 

\begin{example}
For a ring~$R$ and left $R$-modules~$M$ and~$N$ (also considered as complexes concentrated in degree zero) we have the following identification: 
$$\hhom_{\DD_R(e)}(M,N)^n\quad=\quad\hhom_{D(R)}(\Sigma^{-n}M,N)\quad=\quad\Ext^n_R(M,N)$$
\end{example}
\noindent
Let us now come to a graded-commutative variant of the center for stable derivators.

\begin{definition}
Let $\DD$ be a stable derivator and let $\Sigma\colon\DD\nach\DD$ be the suspension morphism. Then the \emph{graded center} $\Z_\bullet(\DD)$ \emph{of} $\DD$ is the $\mathbb{Z}-$graded abelian group which in degree $n$ is the subgroup $\Z_n(\DD)=\Z^{-n}(\DD)$ of $\nat(\Sigma^n,\id_{\DD})$ given by the natural transformations $\tau$ that commute with the suspension morphism up to a sign, i.e., such that:
$$\Sigma\tau\quad=\quad(-1)^n\tau\Sigma\colon\quad\Sigma^{n+1}\nach\Sigma$$
\end{definition}

It is immediate to see that the composition of elements of the graded center endows $\Z_\bullet(\DD)$ with the structure of a graded-commutative ring. Similarly to the unstable case, we can now talk about graded-linear structures. A \emph{graded-linear structure on a stable derivator}~$\DD$ is a map $\sigma\colon R_\bullet\nach\Z_\bullet(\DD)$ of graded-commutative rings. Similarly to the ungraded case, it follows that also the homotopy Kan extensions are $R_\bullet$-linear.

\begin{lemma}
Let $\DD$ be a stable derivator endowed with a graded-linear structure over~$R_\bullet$ and let $u\colon J\nach K$ be a functor. Then the values of~$\DD$ are canonically $R_\bullet$-linear categories and also the induced functors $u^\ast,\:u_!,$ and $u_\ast$ are linear over~$R_\bullet.$
\end{lemma}

Also additive derivators can be considered as graded derivators and admit a notion of a graded center. This straightforward generalization is left to the reader.

\subsection{Linear structures on additive, monoidal derivators}

Let us now turn towards the linear structures which are canonically available for monoidal, additive derivators. The 2-categorical Yoneda lemma (its construction will be recalled in the following proof) gives us for every monoidal prederivator $\EE$ the strict morphism $\kappa_{\SSS_e}\colon e\to\EE$ corresponding to the monoidal unit $\SSS_e$ of the underlying monoidal category $\EE(e).$

\begin{lemma}\label{lem_strictunit}
Let $\EE$ be a monoidal prederivator. Then the unit morphism $\SSS\colon e\to\EE$ and the strict morphism $\kappa_{\SSS_e}\colon e\to\EE$ are naturally isomorphic.
\end{lemma}
\begin{proof}
Recall from the proof of the 2-Yoneda lemma that the value of $\kappa_{\SSS_e}$ at a category $K$ is just the element $p_K^\ast(\SSS_e)$ where $p_K\colon K\to e$ is the unique functor to the terminal category $e.$ Moreover, for a functor $u\colon J\nach K$ the induced functors $u^\ast\colon\EE(K)\to\EE(J)$ are canonically monoidal functors. In particular, there is a canonical isomorphism $u^\ast(\SSS_K)\to\SSS_J$ which is, as we saw in the first subsection, the structure isomorphism $\gamma_u$ belonging to the morphism of prederivators $\SSS\colon e\to\EE.$ Applied to the canonical functor $p_K,$ this gives us an isomorphism $$\tau_K=\gamma_{p_K}\colon(\kappa_{\SSS_e})_K=p_K^\ast(\SSS_e)\nach\SSS_K.$$
These $\tau_K$ assemble to a natural isomorphism $\tau\colon\kappa_{\SSS_e}\to\SSS.$ In fact, we just have to check that the following diagram commutes:
$$
\xymatrix{
u^\ast p_K^\ast(\SSS_e)\ar[rrr]^{u^\ast\tau_K=u^\ast\gamma_{p_K}}\ar@{=}[d]&&&u^\ast\SSS_K\ar[d]^{\gamma_u}\\
p_J^\ast(\SSS_e)\ar[rrr]_{\tau_Ju^\ast=\gamma_{p_J}u^\ast}&&&\SSS_J
}
$$
But this is just a special case of the coherence properties of the isomorphisms belonging to the morphisms of prederivators $\SSS\colon e\nach\EE.$ 
\end{proof}

We can also give a more conceptual proof of this lemma. For this purpose, let us recall the \emph{bicategorical} Yoneda lemma. For a general introduction to the theory of bicategories see \cite{benabou}. Although we are only concerned with 2-categories, let us quickly mention that the basic idea with bicategories is that one wants to relax the notion of 2-categories in the sense that one only asks for a composition law which is unital and associative up to specified natural coherent isomorphisms. Given two 2-categories $\mc$ and $\md$ and two parallel 2-functors $F,G\colon\mc\nach\md$ one can now consider the category $\PsNat(F,G)$ of pseudo-natural transformations where the morphisms are given by the modifications. As a special case, let us take $\md=\CAT$, let us fix an object $X\in\mc$ and let us consider the corepresented 2-functor $y(X)=\Hom(X,-):\mc\nach\CAT.$ If we are given in addition a $\CAT$-valued 2-functor $F\colon\mc\nach\CAT$ then we can consider the category $\PsNat(y(X),F)$. The bicategorical Yoneda lemma states that the evaluation at the identity of $X$ induces a natural \emph{equivalence of categories}\label{bicategoricalYoneda}:
$$Y\colon\PsNat(y(X),F)\stackrel{\simeq}{\nach} F(X)$$ 
The bicategorical Yoneda lemma in the more general situation of homomorphisms of bicategories can be found in \cite{street_fibinbicats}.

Let us recall that given two prederivators $\DD$ and $\DD'$ the category of morphisms from $\DD$ to $\DD'$ is given by $\Hom(\DD,\DD')=\PsNat(\DD,\DD').$ The bicategorical Yoneda lemma hence gives us an equivalence of categories
$$Y\colon\Hom (e,\DD)=\Hom(y(e),\DD)\stackrel{\simeq}{\nach}\DD(e).$$
In the special case where $\DD=\EE$ is monoidal the above lemma follows since both morphisms $\SSS$ and $\kappa_{\SSS_e}$ are mapped to $\SSS_e$ under $Y$ showing that they must be isomorphic.

Let now~$\EE$ be a monoidal, additive derivator. By definition of a monoidal derivator the monoidal pairing preserves homotopy colimits separately in each variable so that it is in particular biadditive. Conjugation by the natural isomorphism of Lemma \ref{lem_strictunit} induces the third map in the following composition of ring homomorphisms:
$$
\hhom_{\EE(e)}(\SSS_e,\SSS_e)\to\nat(\kappa_{\SSS_e},\kappa_{\SSS_e})\to\nat(\kappa_{\SSS_e}\otimes-,\kappa_{\SSS_e}\otimes-)\to\nat(\SSS\otimes -,\SSS\otimes -)
$$
A final conjugation with the coherence isomorphism $l\colon\mathbb{S}\otimes-\cong\id$ thus gives us a ring map
$$
\sigma\colon\hhom_{\EE(e)}(\SSS_e,\SSS_e)\nach \Z(\EE),
$$
i.e., the derivator $\EE$ is endowed with a linear structure over the endomorphisms of $\mathbb{S}_e.$ Thus, we have proved the following result.

\begin{proposition}\label{prop_linear}
A monoidal, additive derivator~$\EE$ is  canonically linear over $\hhom_{\EE(e)}(\SSS_e,\SSS_e).$ 
\end{proposition}

\begin{corollary}
Let~$\mm$ be a combinatorial, closed monoidal model category with monoidal unit~$\SSS.$ If the derivator~$\DD_{\mm}$ is additive then it is canonically linear over $\hhom_{\Ho(\mm)}(\SSS,\SSS).$
\end{corollary}

We will mention a few specific examples at the end of this subsection. Let us note that there is a certain asymmetry in the construction of the above linear structures: we only used the coherence isomorphism $\SSS\otimes -\cong\id.$ As a consequence, a similar result concerning the existence of linear structures can also be established for additive derivators which are modules over monoidal, additive derivators via an additive action (cf.\ \cite{groth_enriched}). Furthermore, there is a graded variant of this result for stable, monoidal derivators. 

But before we come to that we want to mention that the existence of these linear structures is only the shadow of a much more structured result. In \cite{groth_enriched} we will see that a derivator has an associated derivator of endomorphisms denoted $\END(\DD).$ In that paper we establish a general 2-categorical result implying that $\END(\DD)$ is canonically monoidal and gives the (bi)terminal example of a monoidal derivator acting on~$\DD$. In the case of a monoidal derivator $\EE$ this specializes to the existence of a canonical monoidal morphism $\EE\nach\END(\EE).$ In the additive context, the ring map of Proposition~\ref{prop_linear} is just a shadow of this monoidal morphism of derivators.

Recall, e.g.\ from \cite[Definition A.2.1]{hps_axiomatic} and \cite[Section 4]{may_additivity}, that there are notions of when a closed monoidal structure on a triangulated category is \emph{compatible} with the triangulation. In the context of stable derivators the `triangulation' is not an additional structure. Moreover, some of the compatibility assumptions made in loc.cit.\ are automatically satisfied for stable, (closed) monoidal derivators. For example the monoidal pairing commutes with homotopy colimits separately in both variables and hence also with suspensions. In the symmetric case we can consider for arbitrary $s,\:t\in\mathbb{Z}$ the following diagram:
$$
\xymatrix{
\Sigma^s\SSS\otimes\Sigma^t\SSS\ar[d]_\cong\ar[r]^-\cong&\Sigma^{s+t}\SSS\ar[d]^{(-1)^{st}}\\
\Sigma^t\SSS\otimes\Sigma^s\SSS\ar[r]_-\cong&\Sigma^{t+s}\SSS
}
$$
The vertical map $(-1)^{st}$ makes of course sense since a stable derivator is, in particular, additive. 

\begin{lemma}\label{lem_compatible}
Let~$\EE$ be a symmetric monoidal derivator then the above square commutes for all~$s,\:t\in\mathbb{Z}$ in the case of a stable derivator and for all~$s,\:t\in\mathbb{N}$ in the case of an additive derivator.
\end{lemma}
\begin{proof}
First, it is enough to prove that this holds true in the case of $s=t=1.$ Then, the universal property of the derivator~$\DD_{\mathsf{Top}_\ast}$ associated to pointed topological spaces guarantees that every pointed derivator is canonically tensored over this monoidal derivator. Moreover, since this action preserves homotopy colimits in the first variable it follows that the suspension morphism on~$\EE$ is given by the action of the sphere $S^1\in\Ho(\mathsf{sSet}_\ast).$ The fact that the twist map $t\colon S^1\wedge S^1\cong S^1\wedge S^1$ has degree $-1$ concluded the proof.
\end{proof}

\begin{lemma}
Let~$\EE$ be a stable, symmetric monoidal derivator and let~$\SSS_e$ be the monoidal unit of the underlying monoidal category~$\EE(e).$ Then the graded ring of self-maps  $\hhom_{\EE(e)}(\SSS_e,\SSS_e)_\bullet$ is graded-commutative.
\end{lemma}
\begin{proof}
As a special case of the composition in the graded category~$\EE(e),$ the composition of two graded morphisms $g\colon\Sigma^t\SSS_e\nach\SSS_e$ and $f\colon\Sigma^s\SSS_e\nach\SSS_e$ is given by:
$$g\circ f\colon\quad\Sigma^{t+s}\SSS_e\stackrel{\Sigma^t f}{\nach}\Sigma^t\SSS_e\stackrel{g}{\nach}\SSS_e$$
The graded-commutativity of this composition is now implied by the following diagram which is commutative by Lemma~\ref{lem_compatible}:
$$
\xymatrix{
\Sigma^s\SSS_e\ar@/_3.0pc/[dd]_-f&\Sigma^{s+t}\ar[l]_-{\Sigma^s g}\SSS_e\ar[r]^{(-1)^{st}}&\Sigma^{t+s}\ar[r]^-{\Sigma^tf}\SSS_e&\Sigma^t\SSS_e \ar@/^3.0pc/[dd]^-g\\
\Sigma^s\SSS_e\otimes\SSS_e\ar[u]_\cong\ar@/_0.3pc/[dr]_-{f\otimes\id}& \Sigma^s\SSS_e\otimes\Sigma^t\SSS_e\ar[u]^\cong\ar[l]^-{\id\otimes g}\ar[r]_-\cong\ar[d]^{f\otimes g}&
\Sigma^t\SSS_e\otimes\Sigma^s\SSS_e\ar[u]_\cong\ar[r]_-{\id\otimes f}\ar[d]_{g\otimes f}& \Sigma^t\SSS_e\otimes\SSS_e\ar[u]^\cong\ar@/^0.3pc/[dl]^-{g\otimes\id}\\
\SSS_e&\SSS_e\otimes\SSS_e\ar[r]_-\cong\ar[l]^-\cong&\SSS_e\otimes\SSS_e\ar[r]_-\cong&\SSS_e
}
$$
Here, the `composition of the bottom line' just gives~$\id_{\SSS_e}$ by one of the coherence axioms for a symmetric monoidal category.
\end{proof}

\begin{proposition}
A stable, symmetric monoidal derivator~$\EE$ is canonically endowed with a linear structure over the graded-commutative ring $\hhom_{\EE(e)}(\SSS_e,\SSS_e)_\bullet,$ i.e., we have a morphism of graded rings
$$\hhom_{\EE(e)}(\SSS_e,\SSS_e)_\bullet\nach \Z_\bullet(\EE).$$
\end{proposition}
\begin{proof}
We only give a sketch of the proof. Using the same notation as in the unstable case, we obtain a map
$\hhom_{\DD(e)}(\SSS_e,\SSS_e)_n=\hhom_{\DD(e)}(\Sigma^n\SSS_e,\SSS_e)\nach\nat(\kappa_{\Sigma^n\!\SSS_e}\otimes-,\kappa_{\SSS_e}\otimes-)$
which can be composed with the following chain of identifications:
\begin{eqnarray*}
\nat(\kappa_{\Sigma^n\!\SSS_e}\otimes-,\kappa_{\SSS_e}\otimes-)&\cong&\nat(\Sigma^n\circ(\kappa_{\SSS_e}\otimes-),\kappa_{\SSS_e}\otimes-)\\
&\cong& \nat(\Sigma^n\circ (\SSS\otimes-),\SSS\otimes-)\\
&\cong& \nat(\Sigma^n,\id_{\DD})\\
&=& \Z_n(\DD)
\end{eqnarray*}
These assemble together to define the intended map of graded rings $\hhom_{\DD(e)}(\SSS_e,\SSS_e)_\bullet\nach \Z_\bullet(\DD).$ 
\end{proof}

This can be applied to interesting derivators associated to combinatorial, stable, monoidal model categories. We again take up two of the examples of Subsection \ref{subsection_monoidalmodel} which give rise to stable, closed monoidal derivators. In the sequel~\cite{groth_enriched} to this paper these examples will be continued to include interesting graded-linear structures on derivators of modules over not necessarily commutative monoids in the respective context.

\begin{example}
Let us consider the projective model structure on the category~$\Ch(k)$ of chain complexes over~$k$. The ring of endomorphisms of the monoidal unit~$k[0]$ in the homotopy category, i.e., in the derived category~$D(k)$ of the ring~$k$, is just the ground ring, i.e., we have $\hhom_{D(k)}(k[0],k[0])\cong k$. Thus, the derivator $\DD_k$ is canonically endowed with a $k$-linear structure. Since this model structure is stable we also obtain a graded-linear structure. In this example this does not lead to an interesting additional structure since the graded ring of endomorphisms $\hhom_{D(k)}(k[0],k[0])_\bullet$ is concentrated in degree zero.

But the graded-linear structure is more interesting in the case of the stable, monoidal derivator~$\DD_C$ associated to a commutative differential-graded algebra~$C$ over~$k$. The monoidal unit is~$C$ itself and its graded ring of self-maps in $\DD_C(e)=\Ho(\Mod-C)$ is canonically isomorphic to the homology~$H_\bullet(C)$. Thus,~$\DD_C$ is endowed with a graded-linear structure $H_\bullet(C)\nach \Z_\bullet(\DD_C).$
\end{example}

\begin{example}
Let us consider the absolute projective stable model structure on the category~$\Sp^\mathsf{\Sigma}$. The endomorphisms of the sphere spectrum in the homotopy category, i.e., in \emph{the stable homotopy category}~$\SHC,$ are the integers, i.e., we have $\hhom_{\SHC}(\SSS,\SSS)\cong\mathbb{Z}.$ Thus, the derivator~$\DD_{\Sp}$ is canonically $\mathbb{Z}$-linear what we already knew since it is stable and hence, in particular, additive. But there is more structure: the graded self-maps of the sphere spectrum in~$\SHC$ form the graded-commutative ring~$\pi^S_\bullet$ given by the stable homotopy groups of spheres. Thus, the derivator~$\DD_{\Sp}$ is endowed with a graded-linear structure $\pi^S _\bullet\nach \Z_\bullet (\DD_{\Sp})$. In particular, all categories $\DD_{\Sp}(K)$ are $\pi^S_\bullet$-linear categories and all induced functors~$u^\ast,\:u_!,$ and~$u_\ast$ preserve these linear structures. 

Similarly, if~$E$ is a commutative ring spectrum, then the derivator~$\DD_E$ of right $E$-module spectra is canonically endowed with a linear structure over the graded ring of self-maps of~$E$ in the homotopy category $\Ho(\Mod-E).$ Thus, we obtain a graded ring map $\pi_\bullet(E)\nach \Z_\bullet(\DD_E)$ where $\pi_\bullet(E)$ denotes the graded-commutative ring of homotopy groups of~$E.$
\end{example}

\newpage

\bibliographystyle{alpha}
\bibliography{derivators_2}

\def\cprime{$'$}
\begin{thebibliography}{DHKS04}

\bibitem[AR94]{adamekrosicky}
Ji{\v{r}}{\'{\i}} Ad{\'a}mek and Ji{\v{r}}{\'{\i}} Rosick{\'y}.
\newblock {\em Locally presentable and accessible categories}, volume 189 of
  {\em London Mathematical Society Lecture Note Series}.
\newblock Cambridge University Press, Cambridge, 1994.

\bibitem[Bel01]{beligiannis_homotopy}
Apostolos Beligiannis.
\newblock Homotopy theory of modules and {G}orenstein rings.
\newblock {\em Math. Scand.}, 89(1):5--45, 2001.

\bibitem[B{\'e}n67]{benabou}
Jean B{\'e}nabou.
\newblock Introduction to bicategories.
\newblock In {\em Reports of the {M}idwest {C}ategory {S}eminar}, pages 1--77.
  Springer, Berlin, 1967.

\bibitem[BM94]{beligiannis_left}
Apostolos Beligiannis and Nikolaos Marmaridis.
\newblock Left triangulated categories arising from contravariantly finite
  subcategories.
\newblock {\em Comm. Algebra}, 22(12):5021--5036, 1994.

\bibitem[Bor94a]{borceux1}
Francis Borceux.
\newblock {\em Handbook of categorical algebra. 1}, volume~50 of {\em
  Encyclopedia of Mathematics and its Applications}.
\newblock Cambridge University Press, Cambridge, 1994.
\newblock Basic category theory.

\bibitem[Bor94b]{borceux2}
Francis Borceux.
\newblock {\em Handbook of categorical algebra. 2}, volume~51 of {\em
  Encyclopedia of Mathematics and its Applications}.
\newblock Cambridge University Press, Cambridge, 1994.
\newblock Categories and structures.

\bibitem[BR07]{beligiannisreiten}
Apostolos Beligiannis and Idun Reiten.
\newblock Homological and homotopical aspects of torsion theories.
\newblock {\em Mem. Amer. Math. Soc.}, 188(883):viii+207, 2007.

\bibitem[Cis08]{cisinski_derivedkan}
Denis-Charles Cisinski.
\newblock Propri\'et\'es universelles et extensions de {K}an d\'eriv\'ees.
\newblock {\em Theory Appl. Categ.}, 20:No. 17, 605--649, 2008.

\bibitem[CN08]{cisinskineeman}
Denis-Charles Cisinski and Amnon Neeman.
\newblock Additivity for derivator {$K$}-theory.
\newblock {\em Adv. Math.}, 217(4):1381--1475, 2008.

\bibitem[DHKS04]{dhks_homotopy}
William~Gerard Dwyer, Philip~Steven Hirschhorn, Daniel~Marinus Kan, and
  Jeffrey~H. Smith.
\newblock {\em Homotopy limit functors on model categories and homotopical
  categories}, volume 113 of {\em Mathematical Surveys and Monographs}.
\newblock American Mathematical Society, Providence, RI, 2004.

\bibitem[DS95]{dwyerspalinski}
William~Gerard Dwyer and Jan Spali{\'n}ski.
\newblock Homotopy theories and model categories.
\newblock In {\em Handbook of algebraic topology}, pages 73--126.
  North-Holland, Amsterdam, 1995.

\bibitem[DS97]{day_hopfalgebroids}
Brian Day and Ross Street.
\newblock Monoidal bicategories and {H}opf algebroids.
\newblock {\em Adv. Math.}, 129(1):99--157, 1997.

\bibitem[EK66]{eilenberg_closed}
Samuel Eilenberg and G.~Max Kelly.
\newblock Closed categories.
\newblock In {\em Proc. {C}onf. {C}ategorical {A}lgebra ({L}a {J}olla,
  {C}alif., 1965)}, pages 421--562. Springer, New York, 1966.

\bibitem[Fra96]{franke}
Jens Franke.
\newblock Uniqueness theorems for certain triangulated categories with an
  {A}dams spectral sequence, 1996.
\newblock Preprint.

\bibitem[GJ99]{goerss-jardine}
Paul~G. Goerss and John~F. Jardine.
\newblock {\em Simplicial homotopy theory}, volume 174 of {\em Progress in
  Mathematics}.
\newblock Birkh\"auser Verlag, Basel, 1999.

\bibitem[GPS12]{grothpontoshulman_distributors}
Moritz Groth, Kate Ponto, and Michael Shulman.
\newblock The bicategory of distributors of a monoidal derivator, 2012.
\newblock In preparation.

\bibitem[Gro]{grothendieck}
Alexander Grothendieck.
\newblock Les d\'erivateurs.
\newblock \url{http://www.math.jussieu.fr/~maltsin/groth/Derivateurs.html}.
\newblock Manuscript.

\bibitem[Gro10]{groth_infinity}
Moritz Groth.
\newblock A short course on $\infty$-categories.
\newblock \url{http://arxiv.org/abs/1007.2925}, 2010.
\newblock Preprint.

\bibitem[Gro11]{groth_derivator}
Moritz Groth.
\newblock Derivators, pointed derivators, and stable derivators.
\newblock \url{http://www.arxiv.org/abs/1112.3840}, 2011.
\newblock Preprint.

\bibitem[Gro12]{groth_enriched}
Moritz Groth.
\newblock Enriched derivators, 2012.
\newblock In preparation.

\bibitem[GU71]{gabrielulmer}
Peter Gabriel and Friedrich Ulmer.
\newblock {\em Lokal pr\"asentierbare {K}ategorien}.
\newblock Lecture Notes in Mathematics, Vol. 221. Springer-Verlag, Berlin,
  1971.

\bibitem[Hel88]{heller}
Alex Heller.
\newblock Homotopy theories.
\newblock {\em Mem. Amer. Math. Soc.}, 71(383):vi+78, 1988.

\bibitem[Hir03]{hirschhorn}
Philip~Steven Hirschhorn.
\newblock {\em Model categories and their localizations}, volume~99 of {\em
  Mathematical Surveys and Monographs}.
\newblock American Mathematical Society, Providence, RI, 2003.

\bibitem[Hov99]{hovey}
Mark Hovey.
\newblock {\em Model categories}, volume~63 of {\em Mathematical Surveys and
  Monographs}.
\newblock American Mathematical Society, Providence, RI, 1999.

\bibitem[Hov01]{hovey_spectra}
Mark Hovey.
\newblock Spectra and symmetric spectra in general model categories.
\newblock {\em J. Pure Appl. Algebra}, 165(1):63--127, 2001.

\bibitem[HPS97]{hps_axiomatic}
Mark Hovey, John~H. Palmieri, and Neil~P. Strickland.
\newblock Axiomatic stable homotopy theory.
\newblock {\em Mem. Amer. Math. Soc.}, 128(610):x+114, 1997.

\bibitem[HSS00]{HSS}
Mark Hovey, Brooke Shipley, and Jeff Smith.
\newblock Symmetric spectra.
\newblock {\em J. Amer. Math. Soc.}, 13(1):149--208, 2000.

\bibitem[Joy08]{joyal1}
Andr{\'e} Joyal.
\newblock The theory of quasi-categories {I}, to appear, 2008.
\newblock Preprint.

\bibitem[Kel74]{kelly_doctrinal}
Gregory~Maxwell Kelly.
\newblock Doctrinal adjunction.
\newblock In {\em Category {S}eminar ({P}roc. {S}em., {S}ydney, 1972/1973)},
  pages 257--280. Lecture Notes in Math., Vol. 420. Springer, Berlin, 1974.

\bibitem[Kel91]{keller_universal}
Bernhard Keller.
\newblock Derived categories and universal problems.
\newblock {\em Comm. Algebra}, 19:699--747, 1991.

\bibitem[KS05]{kelly_2cat}
Gregory~Maxwell Kelly and Ross Street.
\newblock Review of the elements of 2-categories.
\newblock {\em Repr. Theory Appl. Categ.}, pages vi+137 pp. (electronic), 2005.
\newblock Reprint of the 1982 original [Cambridge Univ. Press, Cambridge;
  MR0651714].

\bibitem[KS06]{schapira}
Masaki Kashiwara and Pierre Schapira.
\newblock {\em Categories and sheaves}, volume 332 of {\em Grundlehren der
  Mathematischen Wissenschaften [Fundamental Principles of Mathematical
  Sciences]}.
\newblock Springer-Verlag, Berlin, 2006.

\bibitem[KV87]{kellervossieck}
Bernhard Keller and Dieter Vossieck.
\newblock Sous les cat\'egories d\'eriv\'ees.
\newblock {\em C. R. Acad. Sci. Paris S\'er. I Math.}, 305(6):225--228, 1987.

\bibitem[Lur09]{HTT}
Jacob Lurie.
\newblock {\em Higher topos theory}, volume 170 of {\em Annals of Mathematics
  Studies}.
\newblock Princeton University Press, Princeton, NJ, 2009.

\bibitem[Mal01]{maltsiniotis1}
Georges Maltsiniotis.
\newblock Introduction \`a la th\'eorie des d\'erivateurs (d'apr\`es
  {G}rothendieck).
\newblock \url{http://people.math.jussieu.fr/~maltsin/textes.html}, 2001.
\newblock Preprint.

\bibitem[Mal07a]{maltsiniotis2}
Georges Maltsiniotis.
\newblock La {$K$}-th\'eorie d'un d\'erivateur triangul\'e.
\newblock In {\em Categories in algebra, geometry and mathematical physics},
  volume 431 of {\em Contemp. Math.}, pages 341--368. Amer. Math. Soc.,
  Providence, RI, 2007.

\bibitem[Mal07b]{maltsiniotis_quillen}
Georges Maltsiniotis.
\newblock Le th\'eor\`eme de {Q}uillen, d'adjonction des foncteurs d\'eriv\'es,
  revisit\'e.
\newblock {\em C. R. Math. Acad. Sci. Paris}, 344(9):549--552, 2007.

\bibitem[May01]{may_additivity}
J.~P. May.
\newblock The additivity of traces in triangulated categories.
\newblock {\em Adv. Math.}, 163(1):34--73, 2001.

\bibitem[ML98]{maclane}
Saunders Mac~Lane.
\newblock {\em Categories for the working mathematician}, volume~5 of {\em
  Graduate Texts in Mathematics}.
\newblock Springer-Verlag, New York, second edition, 1998.

\bibitem[MP89]{makkai}
Michael Makkai and Robert Par{\'e}.
\newblock {\em Accessible categories: the foundations of categorical model
  theory}, volume 104 of {\em Contemporary Mathematics}.
\newblock American Mathematical Society, Providence, RI, 1989.

\bibitem[PS10]{shulmanponto_shadows}
Kate Ponto and Michael Shulman.
\newblock Shadows and traces in bicategories.
\newblock \url{http://arxiv.org/abs/0910.1306}, 2010.
\newblock Preprint.

\bibitem[Qui67]{quillen}
Daniel~Gray Quillen.
\newblock {\em Homotopical algebra}.
\newblock Lecture Notes in Mathematics, No. 43. Springer-Verlag, Berlin, 1967.

\bibitem[Qui73]{quillen_ktheory}
Daniel~Gray Quillen.
\newblock Higher algebraic {$K$}-theory. {I}.
\newblock In {\em Algebraic {$K$}-theory, {I}: {H}igher {$K$}-theories ({P}roc.
  {C}onf., {B}attelle {M}emorial {I}nst., {S}eattle, {W}ash., 1972)}, pages
  85--147. Lecture Notes in Math., Vol. 341. Springer, Berlin, 1973.

\bibitem[Shi07]{shipley_spectradga}
Brooke Shipley.
\newblock {$H\Bbb Z$}-algebra spectra are differential graded algebras.
\newblock {\em Amer. J. Math.}, 129(2):351--379, 2007.

\bibitem[Shu11]{shulman_composites}
Michael Shulman.
\newblock Comparing composites of left and right derived functors.
\newblock {\em New York J. Math.}, 17:75--125, 2011.

\bibitem[SS00]{SchwedeShipley_Alg}
Stefan Schwede and Brooke~E. Shipley.
\newblock Algebras and modules in monoidal model categories.
\newblock {\em Proc. London Math. Soc. (3)}, 80(2):491--511, 2000.

\bibitem[SS03]{schwedeshipley_equivalences}
Stefan Schwede and Brooke Shipley.
\newblock Equivalences of monoidal model categories.
\newblock {\em Algebr. Geom. Topol.}, 3:287--334 (electronic), 2003.

\bibitem[Str72]{street_formalmonads}
Ross Street.
\newblock The formal theory of monads.
\newblock {\em J. Pure Appl. Algebra}, 2(2):149--168, 1972.

\bibitem[Str80]{street_fibinbicats}
Ross Street.
\newblock Fibrations in bicategories.
\newblock {\em Cahiers Topologie G\'eom. Diff\'erentielle}, 21(2):111--160,
  1980.

\end{thebibliography}

\end{document}